\documentclass[a4paper,11pt,twoside]{article}
\usepackage{latexsym,amsfonts,index,amssymb}
\usepackage[centertags]{amsmath}
\usepackage{amstext}
\usepackage{enumerate}
\usepackage{cases}
\usepackage{enumitem}

\def\titlerunning#1{\gdef\titrun{#1}}
\makeatletter
\def\author#1{\gdef\autrun{\def\and{\unskip, }#1}\gdef\@author{#1}}
\def\address#1{{\def\and{\\\hspace*{18pt}}\renewcommand{\thefootnote}{}%
\footnote {#1}}%
\markboth{\autrun}{\titrun}} \makeatother
\def\email#1{e-mail: #1}
\def\subjclass#1{{\renewcommand{\thefootnote}{}%
\footnote{\emph{Mathematics Subject Classification (2010):} #1}}}
\def\keywords#1{\par\medskip
\noindent\textbf{Keywords.} #1}
\def\subrel#1#2{\mathrel{\mathop{#2}\limits_{#1}}}

\setenumerate[1]{label=$(\alph*)$,leftmargin=*}
\setenumerate[2]{label=$(\roman*)$,leftmargin=*}

\numberwithin{equation}{section}

\newcommand{\lp}[1]{\!\stackrel{\!\!\!\!{\scriptscriptstyle{#1}\!\!\!\!}}{{\scriptstyle
(}}\!}
\newcommand{\rp}[1]{\!\stackrel{\!\!\!\!{\scriptscriptstyle{#1}\!\!\!\!}}{{\scriptstyle )}}\!}
\newcommand{\lpp}[1]{\stackrel{\!\!\!\!{\scriptscriptstyle{#1}}\!\!\!\!}{(}}
\newcommand{\rpp}[1]{\stackrel{\!\!\!\!{\scriptscriptstyle{#1}\!\!\!\!}}{)}}

\def\rk#1{\ensuremath{\mathop{\text{rank}}({#1})}}

\def\Z{{\mathbb Z}}
\newcommand{\A}{{\bf A}}
\newcommand{\Se}{{\bf S}}

\newcommand{\G}{{\bf G}}
\newcommand{\Nil}{{\bf N}}
\newcommand{\LG}{{\bf LG}}
\newcommand{\LSl}{{\bf LSl}}
\newcommand{\V}{{\bf V}}
\newcommand{\OAS}{\overline{\Omega}_A{\bf{S}}}
\newcommand{\bk}{\bar\kappa}

\newcommand{\omek}[3] {{{\Omega}}^{#1}_{#2}{\mathbf #3}}

\newtheorem{theorem}{Theorem}[section]
\newtheorem{proposition}[theorem]{Proposition}
\newtheorem{corollary}[theorem]{Corollary}
\newtheorem{fact}[theorem]{Fact}
\newtheorem{lemma}[theorem]{Lemma}
\newenvironment{claim}[1][Claim]{\begin{trivlist}
\item[\hskip \labelsep {\bfseries #1}]}{\end{trivlist}}

\newtheorem{definition}[theorem]{Definition}
\newtheorem{remark}[theorem]{Remark}
\newtheorem{example}[theorem]{Example}
\newenvironment{definition*}{\begin{trivlist}\item[\hskip
    \labelsep{\bf Definition\quad}]}%
  {\hfill\qed\end{trivlist}}
\newenvironment{notation*}{\begin{trivlist}\item[\hskip
    \labelsep{\bf Notation\quad}]}%
  {\end{trivlist}}

  \def\qed{{\unskip\nobreak\hfil\penalty50\hskip .001pt\hbox{}%
      \nobreak\hfil
      \vrule height 1.2ex width 1.1ex depth -.1ex
      \parfillskip=0pt\finalhyphendemerits=0\medbreak}}
\newenvironment{proof}{\begin{trivlist}\item[\hskip%
     \labelsep{\bf Proof.\quad}]}%
 {\hfill\qed\rm\end{trivlist}}

  {\hfill\qed\end{trivlist}}%

\begin{document}
 \titlerunning{Canonical forms for free $\kappa$-semigroups}

\title{Canonical forms for free $\kappa$-semigroups}

\author{Jos\'{e} Carlos Costa}

\date{January 6, 2014}

\maketitle

\address{Jos\'{e} Carlos Costa: CMAT, Dep.\ Matem\'{a}tica e Aplica\c{c}\~{o}es, Universidade do Minho, Campus
  de Gualtar, 4700-320 Braga, Portugal; %
  \email{jcosta@math.uminho.pt}}


\begin{abstract}
The implicit signature $\kappa$ consists of the multiplication and the $(\omega-1)$-power. We describe a
procedure to transform each $\kappa$-term over a finite alphabet $A$ into a certain canonical form and show that
different canonical forms have different interpretations over some finite semigroup.  The procedure of
construction of the canonical forms, which is inspired in McCammond's normal form algorithm for $\omega$-terms
interpreted over the pseudovariety $\A$ of all finite aperiodic semigroups, consists in applying elementary
changes determined by an elementary set $\Sigma$ of pseudoidentities. As an application, we deduce that the
variety of $\kappa$-semigroups generated by the pseudovariety $\Se$ of all finite semigroups is defined by the
set $\Sigma$ and that the free $\kappa$-semigroup generated by the alphabet $A$ in that variety has decidable word problem.
Furthermore, we show that each $\omega$-term has a unique $\omega$-term in canonical form with the same value
over $\A$. In particular, the canonical forms provide new, simpler, representatives for $\omega$-terms
interpreted over that pseudovariety.

\end{abstract}

\section{Introduction}
A $\kappa$-term is a formal expression obtained from the letters of an alphabet $A$ using two operations: the
binary, associative, concatenation and the unary $(\omega-1)$-power. Instead of working only with
$\kappa$-terms, we will operate in a larger set $T_A^{\bk}$ of terms, called $\bk$-terms
 in~\cite{Almeida&Costa&Zeitoun:2013} (in which $\bk$ is called the \emph{completion} of $\kappa$), obtained from $A$ using
the binary concatenation and the unary $(\omega+q)$-power for each integer $q$. Any $\bk$-term can be given a
natural interpretation on each finite semigroup $S$: the concatenation is interpreted as the semigroup
multiplication while the $(\omega+q)$-power is the unary operation which sends each element $s$ of $S$ to:
$s^{\omega}$, the unique idempotent power of $s$, when $q=0$;
 $s^{\omega}s^{q}$, denoted $s^{\omega+q}$, when $q>0$; the inverse of $s^{\omega-q}$ in the maximal subgroup containing
$s^\omega$, when $q<0$. For a class $\mathcal{C}$ of finite semigroups and $\bk$-terms $\alpha$ and $\beta$, we
say that $\mathcal{C}$ satisfies the $\bk$-identity $\alpha=\beta$, and write $\mathcal{C}\models\alpha=\beta$,
if $\alpha$ and $\beta$ have the same interpretation over every semigroup of~$\mathcal{C}$. The $\bk$-word
(resp.\ $\kappa$-word) problem for $\mathcal{C}$ consists in deciding, given a $\bk$-identity (resp.\ a
$\kappa$-identity) $\alpha=\beta$, whether $\mathcal{C}\models\alpha=\beta$. The $\kappa$-word problem for
$\mathcal{C}$ is certainly a subproblem of the $\bk$-word problem for $\mathcal{C}$. Conversely, each finite
semigroup verifies $x^{\omega+q}=x^{\omega-1}x^{q+1}$ for $q\geq 0$ and $x^{\omega+q}=(x^{\omega-1})^{-q}$ for
$q<0$. This means that for each $\bar\kappa$-term there exists a well determined $\kappa$-term with the same
interpretation over every finite semigroup. As a consequence, the word problems for $\kappa$-terms and for
$\bk$-terms over $\mathcal{C}$ are  equivalent problems.

A pseudovariety of semigroups is a class of finite semigroups closed
under taking subsemigroups, homomorphic images and finite direct
products. We also remember that $\kappa$ and $\bk$ are instances of
so called implicit signatures~\cite{AlmeidaSteinberg:2000}, that is,
sets of implicit operations on finite semigroups containing the
multiplication. A motivation to prove the decidability of the
$\sigma$-word problem, for an implicit signature $\sigma$ and a
pseudovariety ${\bf V}$, is that this is one of the properties
required for ${\bf V}$ to be a $\sigma$-\emph{tame} pseudovariety.
The tameness property of pseudovarieties was introduced by
Almeida and  Steinberg~\cite{AlmeidaSteinberg:2000} with the
purpose of solving the decidability problem for iterated semidirect
products of pseudovarieties. Although that objective has not yet
been reached, tameness has proved to be of interest to solve
membership problems involving other types of
operators~\cite{Almeida&Costa&Zeitoun:2005}. For pseudovarieties of
aperiodic semigroups it is common to use the signature $\omega$,
consisting of the multiplication and the $\omega$-power. A solution
to the $\omega$-word problem has been obtained for the pseudovariety ${\bf A}$ of all finite aperiodic semigroups~\cite{McCammond:2001,Zhiltsov:1999} as well as for some of its most important subpseudovarieties such as ${\bf J}$ of ${\cal J}$-trivial
semigroups~\cite{Almeida:1990}, ${\bf LSl}$ of local
semilattices~\cite{Costa:2001} and
${\bf R}$ of ${\cal R}$-trivial
semigroups~\cite{Almeida&Zeitoun:2007}. For non-aperiodic examples,
in which the $\omega$-power is not enough, we refer to the
pseudovarieties ${\bf CR}$ of completely regular
semigroups~\cite{Almeida&Trotter:2001} and ${\bf LG}$ of local
groups~\cite{Costa&Nogueira&Teixeira:2013b} for which the
$\kappa$-word problem is solved.

In this paper, we study the $\bk$-word problem (and so,
equivalently, the $\kappa$-word problem) for the pseudovariety ${\bf
S}$ of all finite semigroups. A positive solution to this problem has been announced and outlined by Zhil'tsov in~\cite{Zhiltsov:2000} but, unfortunately, the author died without publishing a full version of that note. Our approach is completely independent and consists of three stages. First, we declare some
elements of $T_A^{\bk}$ to be in a certain \emph{canonical form}.
Next, we show that an arbitrary $\bk$-term can be algorithmically
transformed into one in canonical form and with the same value over
${\bf S}$. Finally, we prove that distinct canonical forms have
different interpretations on some finite semigroup. This shows that
for each $\bk$-term there is exactly one in canonical form with the
same value over $\Se$. To test whether a $\bk$-identity
$\alpha=\beta$ holds over $\Se$, it then suffices to verify if the
canonical forms of the $\bk$-terms $\alpha$ and $\beta$ are equal,
thus proving the decidability of the word problem for
$\omek{\bk}{A}{\Se}$, the free $\bk$-semigroup on $A$, via the
homomorphism of $\bk$-semigroups $\eta:T_A^{\bk}\rightarrow
\omek{\bk}{A}{\Se}$ that sends each $a\in A$ to itself. The
canonical forms we use, as well as the procedure of their
construction, are close to the normal forms introduced by
McCammond~\cite{McCammond:2001}  for $\omega$-terms over ${\bf A}$.
For this reason, we adopt some of McCammond's terminology. The proof
of correctness of our algorithm is achieved by associating to each
$\bk$-term $\alpha$ a family of regular languages $L_{{\mathtt
n},{\mathtt p}}(\alpha)$, where ${\mathtt n}$ and ${\mathtt p}$ are
positive integers. The key property is that, if $\alpha$ and $\beta$
are $\bk$-terms in canonical form such that $L_{{\mathtt n},{\mathtt
p}}(\alpha)\cap L_{{\mathtt n},{\mathtt p}}(\beta)\neq\emptyset$ for
large enough ${\mathtt n}$ and ${\mathtt p}$, then $\alpha=\beta$.
This approach is similar to the one followed by Almeida and Zeitoun
in collaboration with the author~\cite{Almeida&Costa&Zeitoun:2012}
to give an alternative proof of correctness over $\A$ of McCammond's
normal form reduction algorithm for $\omega$-terms.

Denote by $T_A^{\omega}$ the subset of $T_A^{\bk}$ formed by all
$\omega$-terms. The subset of the elements of $T_A^{\omega}$ that
are in canonical form does not coincide with the set of McCammond's
$\omega$-terms in normal form. Although the notions of canonical
form and (McCammond's) normal form for $\omega$-terms are similar,
our definition introduces an essential modification in the
conditions of the normal form. This change makes in general a
canonical form be shorter than its normal form. For instance, the
$\omega$-term $(a^\omega b^\omega)^\omega$ is in canonical form
while its normal form is $(a^\omega abb^\omega ba)^\omega a^\omega
abb^\omega$. Moreover each subterm of a canonical form is also a
canonical form, a property that  is useful in inductive proofs and
that fails for normal forms. Furthermore, we show that
 each $\omega$-term has a unique representative in canonical form with the same interpretation over $\A$.

The paper is organized as follows. In
Section~\ref{section:preliminaries}, we review background material
and set the basic notation for $\bk$-terms. We introduce the
$\bk$-terms canonical form definition in Section 3 and prove some of
their fundamental properties. Section 4 is devoted to the
description of the algorithm to transform any given $\bk$-term into
one in canonical form. The languages $L_{{\mathtt n},{\mathtt
p}}(\alpha)$ and their basic characteristics are determined in
Section 5. In Section 6, we complete the proof of the main results
of the paper. Finally, Section 7 attests the uniqueness of canonical
forms for $\omega$-terms over $\A$.

\section{Preliminaries}\label{section:preliminaries}
In this section we begin by briefly reviewing the main definitions and some facts about combinatorics on words
and profinite semigroups. The reader is referred to~\cite{Lothaire:2002,Almeida:1995,Almeida:2002} for further
details about these topics. We then introduce a representation of $\bk$-terms as well-parenthesized words that
extends the representation of $\omega$-words used by McCammond and set up the basic terminology on these
objects.

\paragraph{\textbf{Words.}} Throughout the paper, we work with a finite alphabet $A$. The free semigroup (resp.\ the
free monoid) generated by $A$ is denoted by $A^+$ (resp.\ $A^*$). An element $w$ of $A^*$ is called a (finite)
word and its length is represented by $|w|$. The empty word is denoted by $1$ and its length is $0$. The
following result is known as Fine and Wilf's Theorem (see~\cite{Lothaire:2002}).
\begin{proposition}\label{prop:fine_wilf}
Let $u,v\in A^+$. If two powers $u^k$ and $v^n$ of $u$ and $v$ have
a common prefix of length at least $|u|+|v|-gcd(|u|,|v|)$, then $u$
and $v$ are powers of the same word.
\end{proposition}

A word is said to be {\em primitive} if it cannot be written in the
form $u^n$ with $n>1$.  We say that two words $u$ and $v$ are {\em
conjugate} if there exist words $w_1,w_2\in A^*$ such that
$u=w_1w_2$ and $v=w_2w_1$. Note that, if $u$ is a primitive word and
$v$ is a conjugate of $u$, then $v$ is also primitive.
 Let a total order be fixed on the alphabet $A$.
A {\em Lyndon word} is a primitive word which is minimal in its
conjugacy class, for the lexicographic order that extends to $A^+$
the order on $A$. For instance, with a binary alphabet $A=\{a,b\}$
such that $a<b$, the Lyndon words until length four are
$a,b,ab,aab,abb,aaab,aabb,abbb$. Lyndon words are characterized as
follows~\cite{Duval:1983}.
\begin{proposition}\label{prop:Lyndon_less_than_suffix}
A word is a Lyndon word if and only if it is strictly less than each
of its proper suffixes.
\end{proposition}
In particular, any Lyndon word is \emph{unbordered}, that is, none
of its proper prefixes is one of its suffixes.

\paragraph{\textbf{Pseudowords and $\sigma$-words.}} We denote by $\OAS$ the free profinite semigroup generated by $A$,
whose elements are called \emph{pseudowords} (also known as
\emph{implicit operations}). The free semigroup $A^+$  embeds in
$\OAS$ and is dense in $\OAS$. Given $x\in\OAS$, the closed
subsemigroup of $\OAS$ generated by $x$ contains a single idempotent
denoted by $x^\omega$, which is the limit of the sequence $x^{n!}$.
More generally, for each $q\in\Z$, we denote by $x^{\omega+q}$ the
limit of the sequence $x^{n!+q}$ (with $n!+q>0$).

An \emph{implicit signature}  is a set $\sigma$ of pseudowords
containing the multiplication.  A \emph{$\sigma$-semigroup} is an
algebra in the signature $\sigma$ whose multiplication is
associative. The $\sigma$-subsemigroup of $\OAS$ generated by $A$ is
denoted by $\Omega_A^\sigma \Se$ and its elements are called
\emph{$\sigma$-words}. It is well known that $\Omega_A^\sigma \Se$
is the free $\sigma$-semigroup on $A$. In this paper, we are
interested in the most commonly used implicit signature $\kappa=\{xy
, x^{\omega-1}\}$, usually called the {\em canonical signature}, and
in its extension $\bar\kappa=\{xy,x^{\omega+q}\mid
q\in\mathbb{Z}\}$.  Although $\kappa$ is properly contained in
$\bar\kappa$,  for each $\bar\kappa$-term there exists a well
determined $\kappa$-term with the same interpretation over $\Se$, as
observed above, since this pseudovariety verifies
$x^{\omega+q}=x^{\omega-1}x^{q+1}$ ($q\geq 0$) and
$x^{\omega+q}=(x^{\omega-1})^{-q}$ ($q<0$). This means that
$\Omega_A^\kappa \Se=\Omega_A^{\bar\kappa} \Se$, whence the
signatures $\kappa$ and $\bar\kappa$ have the same {\em expressive
power} over $\Se$, and that the $\kappa$-word and the
$\bar\kappa$-word problems over this class of semigroups are
equivalent.

From hereon, we will work with the signature $\bar\kappa$ and denote by $T_A$ the set of all $\bar\kappa$-terms.
We do not distinguish between $\bk$-terms that only differ in the order in which multiplications are to be
carried out. Sometimes we will omit the reference to the signature $\bar\kappa$ simply referring to an element
of $T_A$ as a term. For convenience, we allow the empty term which is identified with the empty word.

\paragraph{\textbf{Notation for $\bar\kappa$-terms.}}
McCammond~\cite{McCammond:2001} represents $\omega$-terms over~$A$ as nonempty well-parenthesized words over the
alphabet $A\uplus\{\mathord{(},\mathord{)}\}$, which do not have $(\;\!)$ as a factor. For instance, the
$\omega$-term $(a^\omega ba(ab)^\omega)^\omega$ is represented by the parenthesized word $((a) ba(ab))$.
Following this idea, we represent $\bar\kappa$-terms over~$A$ as nonempty well-parenthesized words over the
alphabet $A_\Z=A\uplus\{\;\! \lp{q}\; ,\;\! \rp{q}\;  : q\in\Z\}$, which do not have $\;\lp{q} \ \;\! \rp{q}\;$
as a factor. Every $\bk$-term over~$A$ determines a unique well-parenthesized word over $A_\Z$ obtained by
replacing each subterm $(*)^{\omega+q}$ by $\;\lp{q}*\rp{q}\;$, recursively. Recall that the \emph{rank} of a
term $\alpha$ is the maximum number $\rk \alpha$ of nested parentheses in it. For example, the $\bk$-term
$(a^{\omega-1} ba(ab)^\omega)^{\omega+5}$ has rank $2$ and is represented by $ \lpp{5}\
\lp{\;-\!1}a\rp{-\!1}ba\lp{0}ab\rp{0}\
 \rpp{5}\;\!$, where the rank 2 parentheses are shown in larger size for a greater clarity in
  the representation of the term. Conversely, the $\bk$-term
associated with such a word is obtained by replacing each matching pair of parentheses $\;\lp{q}*\rp{q}\;$ by
$(*)^{\omega+q}$. We identify $T_A$ with the set of these well-parenthesized words over~$A_\Z$. Throughout the
rest of the paper, we will usually refer to a $\bk$-term meaning its associated word over $A_\Z$. Notice that,
while the set $A_\Z$ is infinite, each term uses only a finite number of its symbols.

\paragraph{\textbf{Lyndon terms.}} Since $\bk$-terms are represented
as well-parenthesized words over~$A_\Z$, each definition on words
extends naturally to $\bk$-terms. In particular, a term is said to
be {\em primitive} if it cannot be written in the form $\alpha^n$
with $\alpha\in T_A$ and $n>1$, and two terms $\alpha$ and $\beta$
are {\em conjugate} if there exist terms $\gamma_1,\gamma_2\in T_A$
such that $\alpha=\gamma_1\gamma_2$ and $\beta=\gamma_2\gamma_1$. In
order to describe the canonical form for $\bk$-terms, we need to fix
a representative element in each conjugacy class of a primitive
term. For that, we extend the order on $A$ to $A_\Z$ by letting $\;\lp{p}\ <\
\lp{q}\ <x<\ \rp{q}\ <\ \rp{p}\;$ for all $x\in A$ and $p,q\in\Z$
with $p<q$. A {\em Lyndon term} is a primitive term that is minimal,
with respect to the lexicographic ordering, in its conjugacy class.
For instance, $aab$, $\;\lp{-\!1}aa\rp{-\!1}b\lp{2}aa\rp{2}b\;$ and
$\;\lp{\;-\!2\;}a\rp{-\!2\;}\ \;\!\lpp{0}\ \;\!
\lp{\;-\!1}a\rp{-\!1}ab\rpp{0}\;$ are Lyndon terms.

\paragraph{\textbf{Portions of a $\bar\kappa$-term.}}
Terms of the form $\;\lp{q}\delta\rp{q}\;$ will be called \emph{limit terms}, and $\delta$ and $q$ will be
called, respectively, its \emph{base} and its \emph{exponent}. Consider a rank $i+1$ $\bk$-term
\begin{equation}\label{eq:fact_rank_i+1_term}
\alpha=\gamma_0\lp{q_{\;\!\!1}}\delta_1\rp{q_{\;\!\!1}}\gamma_1\cdots
    \lp{q_{\;\!\!n}}\delta_n\rp{q_{\;\!\!n}}\gamma_n,
\end{equation}with $\rk {\gamma_j}\leq i$ and $\rk{\delta_k} =i$.
The number $n$, of limit terms of rank $i+1$ that are subterms of $\alpha$, will be called the \emph{lt-length}
of $\alpha$. The $\bk$-terms $\gamma_j$ and $\delta_k$ in~\eqref{eq:fact_rank_i+1_term} will be called the
\emph{primary subterms} of $\alpha$ and each $\delta_k$ will in addition be said to be a \emph{base} of
$\alpha$. The factors of $\alpha$ of the form
$\;\lp{q_{\;\!\!k}}\delta_k\rp{q_{\;\!\!k}}\gamma_k\lp{q_{\;\!\!k\;\!\!+\;\!\!1}}\delta_{k+1}\rp{q_{\;\!\!k\;\!\!+\;\!\!1}}\;$
are called \emph{crucial portions} of $\alpha$. The prefix $\gamma_0\lp{q_{\;\!\!1}}\delta_1\rp{q_{\;\!\!1}}\;$
and the suffix $\;\lp{q_{\;\!\!n}}\delta_n\rp{q_{\;\!\!n}}\gamma_n$ of $\alpha$ will be called respectively the
\emph{initial portion} and the \emph{final portion} of $\alpha$.  The product
$\;\lp{q_{\;\!\!n}}\delta_n\rp{q_{\;\!\!n}}\gamma_n\gamma_0\lp{q_{\;\!\!1}}\delta_1\rp{q_{\;\!\!1}}\;$ of the
 final and initial portions will be called the \emph{circular portion} of $\alpha$. Notice that  the circular
portion of $\alpha$ is a crucial portion of $\alpha^2$ and that, if $\alpha$ is not a primitive term, then its
circular portion is a crucial portion of $\alpha$ itself. A term $\gamma_0\delta_1^{j_1}\gamma_1\cdots
    \delta_n^{j_n}\gamma_n$, obtained from $\alpha$ by replacing each subterm $\;\lp{q_{\;\!\!k}}\delta_k\rp{q_{\;\!\!k}}\;$ by
  $\delta_k^{j_k}$ with ${j_k}\geq 1$, is a rank~$i$ term called an \emph{expansion} of
  $\alpha$. When each exponent ${j_k}$ is greater than or equal to a given
  positive integer $p$, the expansion is called a \emph{$p$-expansion} of
  $\alpha$. The notion of expansion is extended to a rank $0$ term $\beta$ by declaring that $\beta$ is its
  own unique expansion.

\section{Canonical forms for $\bar\kappa$-terms}\label{section:normal_forms}
In this section, we give the definition of canonical form $\bar\kappa$-terms and identify the reduction rules
that will be used, in Section~\ref{section:algorithm_normal_form}, to transform each $\bar\kappa$-term $\alpha$
into a canonical form $\alpha'$, with $\rk{\alpha'}\leq \rk{\alpha}$. A consequence of
Theorem~\ref{theo:canonical_form} below is that the $\bk$-term $\alpha'$ is unique and so we call it \emph{the
canonical form of $\alpha$}.

\paragraph{\textbf{Canonical form definition.}}
The canonical form for $\bk$-terms is defined recursively as follows. \emph{Rank~0 canonical   forms} are the
words from~$A^*$. Assuming that rank~$i$ canonical forms have been defined, a \emph{rank~$i+1$ canonical form
  ($\bk$-term)} is a $\bk$-term $\alpha$ of the form

\begin{equation}\label{eq:fact_normal_form}
\alpha=\gamma_0\lp{q_{\;\!\!1}}\delta_1\rp{q_{\;\!\!1}}\gamma_1\cdots
    \lp{q_{\;\!\!n}}\delta_n\rp{q_{\;\!\!n}}\gamma_n,
\end{equation}
where the primary subterms $\gamma_j$ and $\delta_k$ are $\bk$-terms such that the following conditions hold:
\begin{enumerate}[label=$(c\;\!\!f\!.\arabic*)$]
\item\label{item:cf-1} the $2$-expansion $\gamma_0\delta_1^2\gamma_1\cdots
    \delta_n^2\gamma_n$ of $\alpha$ is a rank~$i$ canonical form;

\item\label{item:cf-2} each base $\delta_k$ of $\alpha$ is a Lyndon term of rank~$i$;

\item\label{item:cf-3} no $\delta_k$ is a suffix of  $\gamma_{k-1}$;

\item\label{item:cf-4} no $\delta_k$ is a prefix of some term $\gamma_{k}\delta_{k+1}^\ell$ with $\ell\geq
0$.
\end{enumerate}

For instance, the rank 1 terms $\;\!ab\lp{0}abb\rp{0}ab\lp{-\!2}a\rp{-\!2}\;$ and $\ \lp{-\!1\;\!}b\rp{-\!1\;}\ \;
\lp{4}a\rp{4}b \lp{1}ab\rp{1}\;$ as well as the rank 2 terms $\;\lp{1}a\rp{1\;\!}\ \;\!\lpp{-\!3\; }\
\;\!\lp{\;\!0}b\rp{0}\ \;\lp{1}a\rp{1\;\!}\ \;\!\rpp{-\!3}\ \;\!\lp{0}b\rp{0}\ \;\lp{2}a\rp{2}b\;\!$ and $\;\lpp{2}\
\;\!\lp{\;-\!1}ab\rp{-\!1\;}\ \;\lp{\;-\!1}a\rp{-\!1\;} b \lp{\;\!0}a\rp{0}b\rpp{2}\ \;\!\lpp{0}\
\;\!\lp{\;-\!1}a\rp{-\!1\;} b \lp{\;\!0}a\rp{0}b\rpp{0}$ are in canonical form. We say that a $\bk$-term is in
\emph{semi-canonical form} if it verifies condition~\ref{item:cf-1} of the canonical form definition. Of course,
all canonical forms and all rank 1 terms are in semi-canonical form. The term $\;\lp{0}a\rp{0\;\!}\
\;\!\lpp{-\!1\; }\ \;\!\lp{\;\!0}b\rp{0}\ \;\lp{0}a\rp{0\;\!}\ \;\lp{\;\!0}b\rp{0}\ \;\lp{0}a\rp{0\;\!}\
\;\!\rpp{-\!1}\ \;\!\lp{0}b\rp{0}\ \;\lp{0}a\rp{0}\ \;\lp{\;\!0}b\rp{0}\ \;\lpp{0}\ \;\!\lp{0}a\rp{0\;\!}\
\;\lp{\;\!0}b\rp{0}\ \;\!\rpp{0}$ constitutes an example of  a semi-canonical form of rank $2$ that is not in
canonical form. Notice that the exponents $q_k$ do not intervene in
conditions~\ref{item:cf-1}--\ref{item:cf-4}, which means that $\alpha$ being or not in (semi-)canonical form is
independent of the $q_k$. That is, if a $\bk$-term $\alpha$ of the form~\eqref{eq:fact_normal_form} is in
(semi-)canonical form, then any $\bk$-term obtained from $\alpha$ by replacing the exponent $q_k$
($k=1,\ldots,n$) by some $q'_k$ is also in (semi-)canonical form.

As one may note, the canonical form definition for crucial portions does not coincide with the one that
McCammond~\cite{McCammond:2001} imposed on crucial portions of $\omega$-terms in normal form. While McCammond's
definition is symmetric relative to the limit terms of the crucial portion and in some cases forces the central
factor to have some copies of the bases adjacent to them, we choose to let each limit term absorb all adjacent
occurrences of its base (even when they overlap the limit term on its right side), a strategy already used by
the author in~\cite{Costa:2001} to solve the $\omega$-word problem for the pseudovariety $\LSl$. This way the
 canonical form definition for crucial portions looses symmetry but determines shorter canonical forms. For instance,
the $\bk$-terms $\alpha_1=(a^\omega b^\omega)^\omega$ and $\alpha_2=a^{\omega-1}ab
b^{\omega-2}ba(a^{\omega-2}abb^{\omega-2}ba)^{\omega-2}a^{\omega-2}abb^{\omega-1}$ have the same interpretation
over $\Se$. With the above canonical form definition, $\alpha_1$ is the canonical form of both $\alpha_1$ and
$\alpha_2$, while with McCammond's alternative, $\alpha_2$ would be their common canonical form. Moreover these
canonical forms have the following nice property that fails, in part, for McCammond's normal forms.

\begin{proposition}\label{prop:canonicalform_vs_subterms}
 The following conditions are equivalent for a term $\alpha$:
\begin{enumerate}
\item\label{item:canonicalform_vs_subterms-a} The term $\alpha$ is in (semi-)canonical form.

\item\label{item:canonicalform_vs_subterms-b} Every subterm of $\alpha$ is in (semi-)canonical form.

\item\label{item:canonicalform_vs_subterms-c} The initial portion,
the final portion  and all of the crucial portions of $\alpha$ are in (semi-)\linebreak canonical form.
\end{enumerate}
\end{proposition}
\begin{proof} The proof is made by induction on
the rank of $\alpha$. For $\rk {\alpha}=0$, the result holds trivially. Let now $\rk {\alpha}=i+1$ and suppose,
by the induction hypothesis, that the proposition holds for $\bk$-terms of rank at most $i$.

To show the implication
\ref{item:canonicalform_vs_subterms-a}$\Rightarrow$\ref{item:canonicalform_vs_subterms-b},
assume that $\alpha$ is in semi-canonical form and that it has the
form~\eqref{eq:fact_normal_form}. Consider a subterm $\beta$ of
$\alpha$ and let us prove that $\beta$ is in semi-canonical form.
Suppose first that $\rk {\beta}= i+1$. In this case $\beta$ is of
the form
$\beta=\gamma'_{j-1}\lp{q_{\;\!\!j}}\delta_j\rp{q_{\;\!\!j}}\gamma_j\cdots
    \lp{q_{\;\!\!k}}\delta_k\rp{q_{\;\!\!k}}\gamma'_k$
    where $1\leq j\leq k\leq n$, $\gamma'_{j-1}$ is a suffix of
    $\gamma_{j-1}$ and $\gamma'_k$ is a prefix of $\gamma_k$. The
    $2$-expansion $\beta_1=\gamma'_{j-1}\delta_j^2\gamma_j\cdots
    \delta_k^2\gamma'_k$  of $\beta$ is a subterm of the $2$-expansion $\alpha_1=\gamma_0\delta_1^2\gamma_1\cdots
    \delta_n^2\gamma_n$ of $\alpha$. As $\alpha$ is in semi-canonical     form, $\alpha_1$ is in canonical form. Now, since $\alpha_1$ is rank $i$ and $\beta_1$ is a subterm of
    $\alpha_1$, we infer from the induction hypothesis that $\beta_1$
    is in canonical form. Hence, $\beta$ is in semi-canonical form.  Note that assuming further that $\alpha$ is in canonical     form, i.e., that $\alpha$ verifies
    conditions~\ref{item:cf-2}--\ref{item:cf-4}, it follows that also $\beta$ verifies those
    conditions whence it is in canonical form. Suppose now that
$\rk {\beta}\leq i$. Then $\beta$ is a subterm of some primary subterm of $\alpha$, whence it is a subterm of
the rank $i$ canonical form $\alpha_1$. By the induction hypothesis, it follows that $\beta$
 is also in semi-canonical form in this case (and it is in canonical form when $\alpha$ is in canonical form).

The implication
\ref{item:canonicalform_vs_subterms-b}$\Rightarrow$\ref{item:canonicalform_vs_subterms-c}
 is obvious, while \ref{item:canonicalform_vs_subterms-c}$\Rightarrow$\ref{item:canonicalform_vs_subterms-a}
  follows easily from the hypothesis~\ref{item:canonicalform_vs_subterms-c} and from the induction hypothesis.
\end{proof}

We say that a $\bk$-term $\alpha$ is in \emph{circular canonical
form} if $\alpha^2$ is in canonical form. The following observation
is an immediate, trivially verifiable, consequence of
Proposition~\ref{prop:canonicalform_vs_subterms}.
\begin{corollary}\label{corol:circular_vs_canonical}
Let $\alpha$ be a $\bk$-term.
\begin{enumerate}
\item\label{item:circular_vs_canonical-a} The term $\alpha$ is in circular canonical form if and only if both $\alpha$  and
its circular portion are in canonical form.

\item\label{item:circular_vs_canonical-b} If $\alpha$ is in circular canonical form then any conjugate of $\alpha$ is also in circular canonical form.

\item\label{item:circular_vs_canonical-c} If $\alpha$ is in semi-canonical form then every base  of $\alpha$ is in circular
 canonical form and the other primary subterms of $\alpha$ are in canonical  form; more generally, for any  subterm $\beta$ of $\alpha$, every base of $\beta$ is in circular canonical form and
 the other primary subterms of $\beta$ are in canonical  form.

\item\label{item:circular_vs_canonical-d} If $\alpha$ is in canonical form and it is not a primitive term, then $\alpha$ is in circular canonical  form.
\end{enumerate}
\end{corollary}

\paragraph{\textbf{Rewriting rules for $\bk$-terms.}} The procedure to transform an arbitrary $\bk$-term into its canonical form, while
retaining its value on finite semigroups, consists in applying elementary changes resulting from reading in
either direction the $\bk$-identities of the following set $\Sigma$ (where $n,p,q\in\Z$ with $n>0$):
\begin{numcases}{}
  (\alpha ^{\omega+p})^{\omega+q}=\alpha ^{\omega+pq},\nonumber\\
 (\alpha ^{n})^{\omega+q}=\alpha ^{\omega+nq},\nonumber \\
 \alpha ^{\omega+p} \alpha ^{\omega+q}=\alpha ^{\omega+p+q},\nonumber \\
 \alpha \alpha ^{\omega+q}  =\alpha ^{\omega+q+1}=  \alpha ^{\omega+q}\alpha,\nonumber \\
(\alpha \beta)^{\omega+q}\alpha =\alpha (\beta\alpha
)^{\omega+q}.\nonumber
\end{numcases}
The types of changes are therefore given by the following rewriting
rules for terms
\begin{xalignat*}{2}
&1.~\lpp{q\;\!}\ \lp{\;\!p}\alpha\rp{p\;\!}\ \rpp{q}\ \ \rightleftarrows\ \ \lp{pq}\alpha\rp{pq}&
&4_L.~\,\alpha \lp{q}\alpha\rp{q}\ \ \rightleftarrows\ \ \; \lp{q\;\!\!+\;\!\!1}\alpha\rp{\;q\;\!\!+\;\!\!1}\\[-1mm]
&2.~ \lp{q}\alpha^n\rp{q}\ \ \rightleftarrows\ \
\lp{nq}\alpha\rp{nq}&
&4_R.~\lp{q}\alpha\rp{q}\alpha\  \rightleftarrows\ \ \; \lp{q\;\!\!+\;\!\!1}\alpha\rp{\;q\;\!\!+\;\!\!1} \\[-1mm]
&3.~ \lp{p}\alpha\rp{p}\ \;\lp{q}\alpha\rp{q}\ \ \rightleftarrows\ \
\;\! \lp{p\;\!\!+\;\!\!q}\alpha\rp{\;p\;\!\!+\;\!\!q}\ \ & &5.~\
\;\,\lp{q}\alpha\beta\rp{q}\alpha\ \rightleftarrows\
\alpha\lp{q}\beta\alpha\rp{q}
\end{xalignat*}
We call the application of a rule of type 1--4 from left to right
(resp.\ from right to left) a \emph{contraction} (resp.\ an
\emph{expansion}) of that type. An application of a rule of type 5,
in either direction, will be called a \emph{shift}. We say that
terms $\alpha$ and $\beta$ are \emph{equivalent}, and denote
$\alpha\sim
 \beta$, if there is a
derivation from $\alpha$ to $\beta$ (that is, there is a finite
sequence of contractions, expansions and shifts  that starts in
$\alpha$ and ends in  $\beta$).

\begin{example}\label{ex:derivation}
Consider the rank 2 canonical form $\delta=b^5a\lpp{3\;\!}\  \lp{0}b\rp{0}a\rpp{3\;}\ \; \lp{-\!5}b\rp{-\!5}\;$.
The rank 3 term $\alpha=\;\lpp{-\!2\;\!}\delta \rpp{-\!2}$ can be rewritten as follows
$$\begin{array}{rl}
\alpha\rightarrow&\hspace*{-2mm} \;\lpp{-\!2\;} b^5a\lpp{3\;}\ \;\!\lp{\;\!\!-\!5}b\rp{-\!5\;} b^5a\rpp{3\;\!}\
\;\! \lp{\;\!-\!5}b\rp{-\!5\;}\ \;\!\rpp{\;-\!2}\ \rightarrow \ \lpp{-\!2\;}\ \lpp{3}
b^5a\lp{\;\!\!-\!5}b\rp{-\!5\;}\ \;\!\rpp{3} b^5a \lp{\;\!\!-\!5}b\rp{-\!5\;}\ \;\!\rpp{\;-\!2}\ \rightarrow \
\lpp{-\!2\;}\ \lpp{4} b^5a\lp{\;\!\!-\!5}b\rp{-\!5\;}\ \;\!\rpp{4}\ \rpp{\;-\!2}\ \rightarrow \ \lpp{-\!8\;}
b^5a\lp{\;\!\!-\!5}b\rp{-\!5\;}\ \;\!
\rpp{\;-\!8}\;\\[1mm]
\rightarrow&\hspace*{-2mm}\;\lpp{-\!9\;}
b^5a\lp{\;\!\!-\!5}b\rp{-\!5\;}\
\;\!\rpp{\;-\!9}b^5a\lp{\;\!\!-\!5}b\rp{-\!5}\ \;\rightarrow \
b^5a\lpp{-\!9\ }\ \;\lp{\;\!-\!5}b\rp{-\!5}b^5a \rpp{\;-\!9\ }\
\lp{\;\!-\!5}b\rp{-\!5}\ \ \rightarrow \ b^5a\lpp{-\!9\ }\
\lp{0}b\rp{0}a \rpp{\;-\!9\ }\ \;\lp{\;\!-\!5}b\rp{-\!5}\ \;
=\alpha'.
\end{array}$$
The first step in this derivation is an expansion of type $4_R$, the second is a shift, the third step is a
contraction of type $4_R$, the fourth is a contraction of type 1, the fifth step is an expansion of type $4_R$,
the sixth  is a shift, and the final step is a contraction of type $4_R$.
\end{example}

Notice that in the above example $\delta$ is a term of the form
$\varepsilon_1\!\lp{3}\beta\rp{3}\varepsilon_2$ such that $\varepsilon_2\varepsilon_1\sim \beta$. Moreover
$\;\lpp{-\!2\;}\varepsilon_1\!\lp{3}\beta\rp{3}\varepsilon_2
\rpp{-\!2}\;=\alpha\sim\alpha'=\varepsilon_1\lp{-\!9}\beta\rp{-\!9}\varepsilon_2$ and $-9=(3+1)(-2)-1$. The
example illustrates the following observation.
\begin{fact}\label{fact:limit_term}
If $\delta$ is a term of the form $\delta=\varepsilon_1\lp{p}\beta\rp{p}\varepsilon_2$ with
$\varepsilon_2\varepsilon_1\sim \beta$, then
$\lpp{q}\delta\rpp{q}\;\sim\varepsilon_1\lp{r}\beta\rp{r}\varepsilon_2$ where $r=(p+1)q-1$.
\end{fact}
\begin{proof} The sequence of equivalences
$$\lpp{q}\varepsilon_1\lp{p}\beta\rp{p}\varepsilon_2\rpp{q}\;\sim\;
\lpp{q}\varepsilon_1\lp{p}\varepsilon_2\varepsilon_1\rp{p}\varepsilon_2\rpp{q}\;\sim\; \lpp{q}\
\lp{p}\varepsilon_1\varepsilon_2\rp{p}\varepsilon_1\varepsilon_2\rpp{q}\;\sim\; \lpp{q\;}\ \;\!\lp{\
p\;\!\!+\;\!\!1}\varepsilon_1\varepsilon_2\rp{p\;\!\!+\;\!\!1\ }\ \;\!\rpp{\;q}\;\sim\  \lp{\
r\;\!\!+\;\!\!1}\varepsilon_1\varepsilon_2\rp{r\;\!\!+\;\!\!1\
}\ \sim\varepsilon_1\lp{r}\varepsilon_2\varepsilon_1\rp{r}\varepsilon_2\sim\varepsilon_1\lp{r}\beta\rp{r}\varepsilon_2$$
can be easily deduced from the hypothesis $\varepsilon_2\varepsilon_1\sim \beta$ and the reduction rules.
\end{proof}

 Since all $\bk$-identities of $\Sigma$ are easily shown to
be valid in $\Se$, if $\alpha\sim \beta$ then $\Se\models \alpha=\beta$. We will prove below that the converse
implication also holds. We do this by transforming each $\bk$-term into an equivalent canonical form and by
showing that, if two given canonical forms are equal over $\Se$ then they are precisely the same $\bk$-term.
This solves the $\bk$-word problem for $\Se$.

\section{The canonical form algorithm}\label{section:algorithm_normal_form}
We describe an algorithm that computes the canonical form of any given $\bk$-term. The algorithm will be defined
recursively on the rank of the given term. Recall first that all rank 0 terms are already in canonical form and
so they coincide with their canonical form. Assuming that the method to determine the canonical form of any term
of rank at most $i$ was already defined, we show below how to reduce an arbitrary term of rank $i+1$ to its
canonical form. The rank $i+1$ canonical form reduction algorithm consists of two major steps. The first step
reduces the given term to a semi-canonical form and the second step completes the calculation of the canonical
form. It will be convenient to start with the description of the second step since this will be used, in rank
$i$, to define the first step in rank $i+1$. Notice that the first step in rank 1 is trivial since every rank 1
term is already in semi-canonical form.

\paragraph{\textbf{Step 2.}}
The procedure to compute the canonical form of an arbitrary  rank $i+1$ term $\alpha_1$ in semi-canonical form
is the following.
\begin{enumerate}[label=2.\arabic*)]
\item\label{ranki+1-step1} Apply all possible rank $i+1$ contractions of type~2.

\item\label{ranki+1-step2} By means of a rank $i+1$ expansion of type $4$, if necessary, and a rank $i+1$ shift,
  write each rank $i+1$ limit term in the form $\;\lp{q}\delta\rp{q}\;$ where $\delta$ is a Lyndon term.

\item\label{ranki+1-step3} Apply all possible rank $i+1$ contractions of type~4.

\item\label{ranki+1-step4} Apply all possible rank $i+1$ contractions of type~3.

\item\label{ranki+1-step5} Put each rank $i+1$ crucial portion $\;\lp{q_{\;\!\!1}}\delta_1\rp{q_{\;\!\!1}}\gamma\lp{\;q_{\;\!\!2}}\delta_2\rp{q_{\;\!\!2}}\;$ in canonical form
as follows. By Step 2.3, $\delta_1$
  is not a prefix and $\delta_2$ is not a suffix of $\gamma$. Let $\ell$ be the minimum nonnegative integer such that
  $|\gamma\delta_2^\ell|\geq |\delta_1|$. If $\delta_1$ is not a prefix of $\gamma\delta_2^\ell$ then the crucial portion
  $\;\lp{q_{\;\!\!1}}\delta_1\rp{q_{\;\!\!1}}\gamma\lp{\;q_{\;\!\!2}}\delta_2\rp{q_{\;\!\!2}}\;$ is already in canonical form.
  Otherwise $\ell\neq 0$. In this case, apply $\ell$ rank $i+1$ expansions of type
  $4_L$ to the limit term on the right side of the crucial portion, followed by all possible, say $n$, rank $i+1$ contractions of type $4_R$, thus obtaining a
  term  $\ \lp{q_{\;\!\!1}\!+\;\!\!n }\delta_1\rp{q_{\;\!\!1}\!+\;\!\!n\;}\;\varepsilon\;\lp{\;\!q_{\;\!\!2}\;\!\!-\;\!\!\ell}
   \delta_2\rp{\;\!q_{\;\!\!2}\;\!\!-\;\!\!\ell}\; $ where $\varepsilon$ is a proper suffix of $\delta_2$ having not a prefix
   $\delta_1$. In view of the following claim the step is complete.
\end{enumerate}

\begin{claim}
 The crucial portion $\ \lp{q_{\;\!\!1}\!+\;\!\!n\;}\delta_1\rp{\;q_{\;\!\!1}\!+\;\!\!n\;}\;\varepsilon\;\lp{\;\!q_{\;\!\!2}\;\!\!-\;\!\!\ell}
   \delta_2\rp{\;\!q_{\;\!\!2}\;\!\!-\;\!\!\ell}\; $ is in canonical form.
\end{claim}
\begin{proof} To prove the claim it suffices to show that $\delta_1$ is not a prefix of $\varepsilon\delta_2^k$ for all $k\geq 1$.
 Assume, by way of contradiction, that $\delta_1$ is a prefix of some $\varepsilon\delta_2^k$. We have $\delta_1=\gamma
\delta_2^{\ell-1}\varepsilon_1$ and $\delta_1^{n-1}=\varepsilon_2$ with
$\delta_2=\varepsilon_1\varepsilon_2\varepsilon$ and $\varepsilon_1$ nonempty. Since $\varepsilon_1$ is a suffix
of the Lyndon term $\delta_1$, we have  $\delta_1\leq \varepsilon_1$ by
Proposition~\ref{prop:Lyndon_less_than_suffix}, and since it is a prefix of $\delta_2$, we have
$\varepsilon_1\leq\delta_2$. Therefore $\delta_1\leq\delta_2$. Suppose that  $\delta_1=\delta_2$. In this case,
$\delta_2\leq \varepsilon_1$ and $\varepsilon_1\leq\delta_2$, whence $\varepsilon_1=\delta_2$. From
$\delta_1=\gamma \delta_2^{\ell-1}\varepsilon_1$ it then follows that $\gamma$ is the empty word (and $\ell=1$).
This is not possible since Step
  2.4 eliminated all crucial portions of the form
  $\;\lp{p}\delta_1\rp{p}\ \ \lp{q}\delta_1\rp{q}\;$. Thus
  $\delta_1\neq\delta_2$ and so $\delta_1<\delta_2$.

Suppose that $\gamma$ is the empty word.  Then $\delta_1=\delta_2^{\ell-1}\varepsilon_1$ and so, as $\delta_2$
cannot be a prefix of $\delta_1$ (since in that case we would have $\delta_2\leq \delta_1$, in contradiction
with $\delta_1<\delta_2$),  $\ell=1$ and $\delta_1= \varepsilon_1$ with $\varepsilon_1$ a proper prefix of
$\delta_2$. Hence $|\delta_1|<|\delta_2|$ and so, from the initial assumption,
$\delta_1=\varepsilon\varepsilon_3$ for some nonempty proper prefix $\varepsilon_3$ of $\delta_2$. As $\delta_1$
is a prefix of $\delta_2$, it follows that $\varepsilon_3$ is both a proper prefix and a suffix of $\delta_1$.
That is, $\delta_1$ is a bordered word, which contradicts the fact of $\delta_1$ being a Lyndon word.
Consequently, we may assume that $\gamma$ is not the empty word.

Suppose next that $n>1$. Then $\varepsilon_2$ is nonempty and so
$|\delta_1|<|\delta_2|$. Hence $\ell=1$ and $\delta_1=\gamma
\varepsilon_1$ with $|\varepsilon_1|<|\delta_1|$, whence
$\varepsilon_1$ is a proper suffix of $\delta_1$ and  a proper
prefix of $\delta_2$. In particular, by
Proposition~\ref{prop:Lyndon_less_than_suffix},
$\delta_1<\varepsilon_1$. On the other hand, as
$\varepsilon_2\varepsilon$ is a proper suffix of $\delta_2$,
$\varepsilon_1<\delta_2<\varepsilon_2\varepsilon=\delta_1^{n-1}\varepsilon$
and thus, as $|\varepsilon_1|<|\delta_1|$, $\varepsilon_1<\delta_1$.
We reached a contradiction and so $n=1$ and $\varepsilon_2$ is
empty.

Suppose now that $\varepsilon$ is the empty word. Then $\varepsilon_1=\delta_2$ is a proper suffix of
$\delta_1$. From the initial assumption it then results that $\delta_2$ is also a prefix of $\delta_1$. This
means that $\delta_1$ is a bordered word, a condition that is impossible because $\delta_1$ is a Lyndon word.
Therefore $\varepsilon$ is a nonempty suffix of the Lyndon word $\delta_2$, whence $\delta_2<\varepsilon$. But
$\varepsilon$ is a proper prefix of $\delta_1$ by the initial assumption and so $\varepsilon<\delta_1$. It
follows that $\delta_2<\delta_1$ in contradiction with the above  inequality $\delta_1<\delta_2$. This shows
that $\delta_1$ cannot be a prefix of some $\varepsilon\delta_2^k$ and proves the claim.
\end{proof}

It is easy to verify that this procedure produces a rank $i+1$ term
$\alpha_2$ in canonical form. Indeed, the reduction rules  that are
eventually used in the process are of type 2--5. Hence, the term
$\alpha_2$ is rank $i+1$ since these rules do not change the rank of
the original term $\alpha_1$. Moreover, by
Proposition~\ref{prop:canonicalform_vs_subterms}, $\alpha_2$ is in
semi-canonical form  since $\alpha_1$ also is and the reduction
rules are all applied in rank $i+1$ and  do not change the initial,
final and crucial portions of the 2-expansions of the term. On the
other hand, Steps 2.1 and 2.2 guarantee that $\alpha_2$ verifies
condition~\ref{item:cf-2} of the canonical form definition, while
condition~\ref{item:cf-3} is obtained in Step 2.3. Finally,
$\alpha_2$ satisfies~\ref{item:cf-4} due to the application of Steps
2.3 to 2.5. For instance, applying the above algorithm to the rank
$2$ semi-canonical form $\alpha_1=\;\lp{0}a\rp{0\;\!}\
\;\!\lpp{-\!1\; }\ \;\!\lp{\;\!0}b\rp{0}\ \;\lp{0}a\rp{0\;\!}\
\;\lp{\;\!0}b\rp{0}\ \;\lp{0}a\rp{0\;\!}\ \;\!\rpp{-\!1}\
\;\!\lp{0}b\rp{0}\ \;\lp{0}a\rp{0}\ \;\lp{\;\!0}b\rp{0}\ \;\lpp{0}\
\;\!\lp{0}a\rp{0\;\!}\ \;\lp{\;\!0}b\rp{0}\ \;\!\rpp{0}\;$ one gets
the following derivation
 $$\begin{array}{rl}
\alpha_1\rightarrow&\hspace*{-2mm}\;\lp{0}a\rp{0\;\!}\ \;\!\lpp{-\!2\; }\ \;\!\lp{\;\!0}b\rp{0}\
\;\lp{0}a\rp{0\;\!}\ \;\!\rpp{-\!2}\ \;\!\lp{0}b\rp{0}\ \;\lp{0}a\rp{0}\ \;\lp{\;\!0}b\rp{0}\ \;\lpp{0}\
\;\!\lp{0}a\rp{0\;\!}\ \;\lp{\;\!0}b\rp{0}\ \;\!\rpp{0}\ \rightarrow\ \;\lp{0}a\rp{0\;\!}\ \;\lp{0}b\rp{0\;\!}\
\;\!\lpp{-\!2\; }\ \;\!\lp{\;\!0}a\rp{0}\ \;\lp{0}b\rp{0\;\!}\ \;\!\rpp{-\!2}\ \;\lp{0}a\rp{0}\
\;\lp{\;\!0}b\rp{0}\ \;\lpp{0}\ \;\!\lp{0}a\rp{0\;\!}\ \;\lp{\;\!0}b\rp{0}\ \;\!\rpp{0}\\[1mm]
\rightarrow&\hspace*{-2mm}\;\lpp{-\!1\; }\ \;\!\lp{\;\!0}a\rp{0}\ \;\lp{0}b\rp{0\;\!}\ \;\!\rpp{-\!1}\
\;\lp{0}a\rp{0}\ \;\lp{\;\!0}b\rp{0}\ \;\lpp{0}\ \;\!\lp{0}a\rp{0\;\!}\ \;\lp{\;\!0}b\rp{0}\ \;\!\rpp{0}\
\rightarrow\ \;\lpp{0}\ \;\!\lp{0}a\rp{0\;\!}\ \;\lp{\;\!0}b\rp{0}\ \;\!\rpp{0}\ \lpp{0}\ \;\!\lp{0}a\rp{0\;\!}\
\;\lp{\;\!0}b\rp{0}\ \;\!\rpp{0}\ \rightarrow\ \; \lpp{0}\ \;\!\lp{0}a\rp{0\;\!}\ \;\lp{\;\!0}b\rp{0}\
\;\!\rpp{0}.
\end{array}$$
The canonical form of $\alpha_1$ is thus the rank $2$ term $\alpha_2=\;\lpp{0}\ \;\!\lp{0}a\rp{0\;\!}\
\;\lp{\;\!0}b\rp{0}\ \;\!\rpp{0}$.

\paragraph{\textbf{Some preliminary remarks to the first step.}}
The rank $i+1$ canonical form reduction algorithm will be completed below with the description of Step 1. For
now we present some preparatory results which will be useful for that purpose.

\begin{lemma}\label{lemma:normalization_of_crucial_portion} Let
$i\geq 1$ and let $\alpha$ be a rank $i$ crucial portion of the form
$\;\lp{q_{\;\!\!1}}\delta_1\rp{q_{\;\!\!1}}\gamma\lp{\;q_{\;\!\!2}}\delta_2\rp{q_{\;\!\!2}}\;$, where the bases
$\delta_1$ and $\delta_2$ are rank $i-1$ Lyndon terms in circular canonical form and $\rk{\gamma}\leq i-1$. The
canonical form $\alpha'$ of $\alpha$ is a rank $i$ term that can be computed using Step $1$ of rank at most
$i-1$ and Step $2$ of rank at most $i$ of the canonical form reduction algorithm. Moreover, either
\begin{enumerate}[label=(\Roman*)]
\item\label{item:normalization_of_crucial_portion:1}  $\alpha'$ is a limit term $\;\lp{r_1}\delta_1\rp{r_1}\;$, in which case
$\delta_1=\delta_2$, the canonical form of $\gamma$ is $\delta^p_1$ for some $p\geq 0$ and $r_1=q_1+q_2+p$; or
\item\label{item:normalization_of_crucial_portion:2} $\alpha'$ is a crucial portion of the form
$\;\lp{r_{\;\!\!1}}\delta_1\rp{r_{\;\!\!1}}\varepsilon\lp{\;\!r_{\;\!\!2}}\delta_2\rp{r_{\;\!\!2}}\;$.
\end{enumerate}
We then say that $\alpha$ is of type~\ref{item:normalization_of_crucial_portion:1}
or~\ref{item:normalization_of_crucial_portion:2} depending on the
condition~\ref{item:normalization_of_crucial_portion:1} or~\ref{item:normalization_of_crucial_portion:2} that
the canonical form $\alpha'$ verifies.
\end{lemma}
\begin{proof} We proceed by induction  on $i$. Assume first that $i=1$. Then $\alpha$ is a
rank 1 term and so it is in semi-canonical form. The canonical form $\alpha'$ may therefore be obtained by the
application to $\alpha$ of Step 2 of the rank 1 canonical form algorithm. Moreover, since $\delta_1$ and
$\delta_2$ are Lyndon terms by hypothesis, Steps 2.1 and 2.2 of the algorithm do not apply. If a contraction of
type 3 is applied in the process (necessarily in Step 2.4), then
 $\delta_1$ and $\delta_2$ are clearly the same word and so, by
 Step 2.3, $\gamma=\delta_1^p$ for
some $p\geq 0$. Hence, $\alpha'=\;\lp{r_1}\delta_1\rp{r_1}\;$ with $r_1=q_1+q_2+p$. If a contraction of type 3
is not applied, then $\alpha'$ is a crucial portion of the form
$\;\lp{r_{\;\!\!1}}\delta_1\rp{r_{\;\!\!1}}\varepsilon\lp{\;\!r_{\;\!\!2}}\delta_2\rp{r_{\;\!\!2}}\;$ since the
transformation process on Step 2.5 does not change the bases $\delta_1$ and $\delta_2$.

Let now $i>1$ and suppose, by the induction hypothesis, that the lemma
holds for crucial portions of rank at most $i-1$. The term $\alpha'$
can be calculated as follows. First, use the rank $j$ canonical form
algorithm, where $j$ is the rank of $\gamma$, to compute the
canonical form $\gamma'$ of $\gamma$. The application of two rank
$i$ expansions of type 4 then give the term
$\;\lp{p_{\;\!\!1}}\delta_1\rp{p_{\;\!\!1}}\delta_1\gamma'\delta_2\lp{\;p_{\;\!\!2}}\delta_2\rp{p_{\;\!\!2}}\;$,
where $p_1=q_1-1$ and $p_2=q_2-1$. By
Proposition~\ref{prop:canonicalform_vs_subterms}, to reduce
$\delta_1\gamma'\delta_2$ to its canonical form $\beta$, it is
sufficient to reduce at most two rank $i-1$ crucial portions to their
canonical form. Indeed, when $\rk{\gamma'}<i-1$, at most the crucial
portion $\pi_1\gamma'\pi_2$ is not in canonical form, where $\pi_1$
is the final portion of $\delta_1$ and $\pi_2$ is the initial
portion of $\delta_2$. If $\rk{\gamma'}=i-1$, then at most the crucial
portions $\pi_1\rho_1$ and $\rho_2\pi_2$ are not in canonical form,
where $\rho_1$ and $\rho_2$ are respectively the initial and the
final portions of $\gamma'$. By condition~\ref{item:cf-2} and
Corollary~\ref{corol:circular_vs_canonical}~\ref{item:circular_vs_canonical-c},
the bases of those crucial portions are rank $i-2$ Lyndon terms in
circular canonical form. Hence, by the induction hypothesis, they
either reduce to a single rank $i-1$ limit term or to another rank $i-1$
crucial portion with the same bases. Therefore, the term
$\delta_1^2\beta\delta_2^2$ is in canonical form
 by Proposition~\ref{prop:canonicalform_vs_subterms}, whence
 $\;\lp{\;p_{\;\!\!1}}\delta_1\rp{p_{\;\!\!1}}\beta\lp{\;p_{\;\!\!2}}\delta_2\rp{p_{\;\!\!2}}\;$ is a rank~$i$
term in semi-canonical form. Since $\delta_1$ and $\delta_2$ are Lyndon terms by hypothesis, to reduce $\;
\lp{p_{\;\!\!1}}\delta_1\rp{p_{\;\!\!1}}\beta\lp{\;p_{\;\!\!2}}\delta_2\rp{p_{\;\!\!2}}\;$ to its canonical form
$\alpha'$ it suffices to apply to it Steps 2.3 to 2.5 of the rank $i$ canonical form algorithm. As in the case
$i=1$ above, one deduces that $\alpha'$ is of one of the forms of the statement, thus completing the inductive
step of the proof.
\end{proof}

The following is the analogue of Lemma~\ref{lemma:normalization_of_crucial_portion} to initial and final
portions. It has a similar proof and so we leave its verification to the reader.

\begin{lemma}\label{lemma:normalization_of_inicial-final_portion} Let
$i\geq 1$ and let $\alpha$ be a rank $i$ initial portion $\gamma\lp{q}\delta\rp{q}\;$ or final portion
$\;\lp{q}\delta\rp{q}\gamma$, where the base $\delta$ is a rank $i-1$ Lyndon term in circular canonical form and
$\rk{\gamma}\leq i-1$. The canonical form $\alpha'$ of $\alpha$ can be computed using Step $1$ of rank at most $i-1$
and Step $2$ of rank at most $i$ of the canonical form reduction algorithm. Moreover,  $\alpha'$ is a rank
$i$ term respectively of the forms  $\varepsilon\lp{\;\!r}\delta\rp{r}\;$ and
$\;\lp{r}\delta\rp{r}\varepsilon$.
\end{lemma}

As a consequence of Proposition~\ref{prop:canonicalform_vs_subterms}
and Lemmas~\ref{lemma:normalization_of_crucial_portion}
and~\ref{lemma:normalization_of_inicial-final_portion}, we get the
following property of the product of two canonical forms.
\begin{lemma}\label{lemma:normalization_of_product}
Let $\alpha$ and $\beta$ be terms in canonical form, let $(\alpha\beta)'$ be the canonical form of $\alpha\beta$
and let $\alpha_1$ be the final portion of $\alpha$ and $\beta_1$ be the initial portion of $\beta$.
\begin{enumerate}
\item\label{item:normalization_of_product-1}  When $\rk{\alpha}<\rk{\beta}$, $(\alpha\beta)'$
is obtained by reducing the initial portion $\alpha\beta_1$ of $\alpha\beta$ to its canonical form. In
particular, the lt-length of $(\alpha\beta)'$ is the lt-length of $\beta$.

\item\label{item:normalization_of_product-2}  When $\rk{\alpha}=\rk{\beta}$, $(\alpha\beta)'$
is obtained by reducing the crucial portion $\alpha_1\beta_1$ of $\alpha\beta$ to its canonical form. In
particular, if $\alpha$ and $\beta$ have lt-length of $m$ and $n$ respectively, then the lt-length of
$(\alpha\beta)'$ is either $m+n-1$ when $\alpha_1\beta_1$ is of
type~\ref{item:normalization_of_crucial_portion:1}, or $m+n$ when $\alpha_1\beta_1$ is of
type~\ref{item:normalization_of_crucial_portion:2}.

\item\label{item:normalization_of_product-3} When $\rk{\alpha}>\rk{\beta}$, $(\alpha\beta)'$
is obtained by reducing the final portion $\alpha_1\beta$ of $\alpha\beta$ to its canonical form. In particular,
the lt-length of $(\alpha\beta)'$ is the lt-length of $\alpha$.
\end{enumerate}
\end{lemma}
\begin{proof} For~\ref{item:normalization_of_product-1}, as $\rk{\alpha}<\rk{\beta}$, the initial portion of
$\alpha\beta$ is $\alpha\beta_1$ and, as $\beta$ is in canonical
form, the base of $\beta_1$ is a Lyndon term in circular canonical
form by condition~\ref{item:cf-2} and
Corollary~\ref{corol:circular_vs_canonical}~\ref{item:circular_vs_canonical-c}.
By Proposition~\ref{prop:canonicalform_vs_subterms}, since all
crucial portions and the final portion of $\alpha\beta$ are from
$\beta$ and $\beta$ is in canonical form, to obtain the canonical
form of $\alpha\beta$ it is sufficient to reduce $\alpha\beta_1$ to
its canonical form. Condition~\ref{item:normalization_of_product-1}
then follows immediately from
Lemma~\ref{lemma:normalization_of_inicial-final_portion}. The proof
of the other conditions is similar.
\end{proof}

Another consequence of the above lemmas
 is the following  property of limit terms, which will be fundamental for the construction of the first step of the canonical form
reduction algorithm.
\begin{proposition}\label{prop:limit_terms_reduction}
Let  $\pi=\ \lp{q}\rho\rp{q}\;$ be a rank $i+1$ limit term with
$i\geq 1$ and base $\rho$ in canonical form. Using the canonical
form reduction algorithm of rank at most $i$, it is possible to
derive from $\pi$ a semi-canonical form $\pi_1$ such that:
\begin{enumerate}
\item\label{item:limit_terms_reduction-a} If $\rho$ has lt-length $1$ and its circular portion is of type
(I), then either:
\begin{enumerate}[label=$(\arabic*)$]
\item $\rho$ is of the form
$\;\lp{q_{\;\!\!1}}\delta_1\rp{q_{\;\!\!1}}\;$ and $\pi_1=\ \;\!
\lp{\;\!qq_{\;\!\!1}}\delta_1\rp{qq_{\;\!\!1}}\;$; or

\item $\rho$ is of the
form $\gamma_0\lp{q_{\;\!\!1}}\delta_1\rp{q_{\;\!\!1}}\gamma_1$ with
$\gamma_1\gamma_0\sim \delta_1$ and
$\pi_1=\gamma_0\lp{r}\delta_1\rp{r}\gamma_1$ where
$r=q(q_{\;\!\!1}+1)-1$.
\end{enumerate}
In both cases $\pi_1$ is a rank $i$ term in canonical form.
\item\label{item:limit_terms_reduction-b} If $\rho$ has lt-length greater than $1$ or its circular portion is of type
(II), then $\pi_1$ is a rank $i+1$ term of the form
$\varepsilon_0\lp{r}\beta\rp{r}\varepsilon_1$ with
$\rk{\varepsilon_0}=\rk{\varepsilon_1}=i$.
\end{enumerate}
\end{proposition}
\begin{proof} Let $\rho=\gamma_0\lp{q_{\;\!\!1}}\delta_1\rp{q_{\;\!\!1}}\gamma_1\cdots\lp{q_{\;\!\!n}}\delta_n\rp{q_{\;\!\!n}}\gamma_n$ be
    the canonical form for $\rho$ and let
    $\alpha=\ \;\lp{q_{\;\!\!n}}\delta_n\rp{q_{\;\!\!n}}\gamma_n\gamma_0\lp{q_{\;\!\!1}}\delta_1\rp{q_{\;\!\!1}}\;$
    be the circular portion of $\rho$. In order to prove~\ref{item:limit_terms_reduction-a}, suppose that $n=1$ and that
    $\alpha$ is of type (I). Then $\rho=\gamma_0\lp{q_{\;\!\!1}}\delta_1\rp{q_{\;\!\!1}}\gamma_1$ and, by
    Lemma~\ref{lemma:normalization_of_crucial_portion}, $\alpha=\
    \;\lp{q_{\;\!\!1}}\delta_1\rp{q_{\;\!\!1}}\gamma_1\gamma_0\lp{q_{\;\!\!1}}\delta_1\rp{q_{\;\!\!1}}\;$
    reduces to a limit term $\;\lp{\;r_1}\delta_1\rp{r_1}\;$ and the canonical form of $\gamma_1\gamma_0$
    is $\delta^p_1$ for some $p\geq 0$. If $p=0$ then $\gamma_0$ and
$\gamma_1$ are both the empty term and so $\pi=\;\lpp{q}\ \lp{\ q_{\;\!\!1}}\delta_1\rp{q_{\;\!\!1}\;}\
\rpp{q}\;$. In this case, applying a contraction of type $1$ one gets a semi-canonical form $\pi_1=\
\lp{\;qq_{\;\!\!1}}\delta_1\rp{qq_{\;\!\!1}}\;$, which is in fact the canonical form of $\pi$. Suppose now that
$p\neq 0$. In this case the rank of $\gamma_1\gamma_0$ is the rank of $\delta_1$, that is,
$\rk{\gamma_1\gamma_0}=i-1$. If $\rk{\gamma_1}<\rk{\gamma_0}$ then, by
Lemma~\ref{lemma:normalization_of_product}~\ref{item:normalization_of_product-1}, the lt-length of $\delta^p_1$
is the lt-length of $\gamma_0$. As $\delta_1$ is not a suffix of $\gamma_0$ by condition~\ref{item:cf-3} of the
canonical form definition, we deduce that $p=1$. The events $\rk{\gamma_1}=\rk{\gamma_0}$ and
$\rk{\gamma_1}>\rk{\gamma_0}$ are treated analogously and give the same result $p=1$. This shows that $\delta_1$
is the canonical form of $\gamma_1\gamma_0$ and, so, that $\gamma_1\gamma_0\sim \delta_1$. By
Fact~\ref{fact:limit_term},  $\pi$ reduces to the term $\gamma_0\lp{r}\delta_1\rp{r}\gamma_1$, where
$r=q(q_{\;\!\!1}+1)-1$, which is obviously in canonical form. This completes the proof
of~\ref{item:limit_terms_reduction-a}.

For~\ref{item:limit_terms_reduction-b}, suppose first  $n>1$. By
Lemma~\ref{lemma:normalization_of_crucial_portion},
    $\alpha$ reduces to a crucial portion of the form
$\;\lp{\;r_{\;\!\!n}}\delta_n\rp{r_{\;\!\!n}}\sigma$, where $\sigma$
is either the empty term (in which case $\delta_n=\delta_1$) or a
term of the form
$\varepsilon\lp{\;r_{\;\!\!1}}\delta_1\rp{r_{\;\!\!1}}\;$. Let
$q'=q-1$ and  $q''=q-2$.  The following derivation is now easily
deduced
 $$\begin{array}{rl}
\pi\
\rightarrow&\hspace*{-2mm}\gamma_0\lp{q_{\;\!\!1}}\delta_1\rp{q_{\;\!\!1}}\gamma_1\cdots\lp{q_{\;\!\!n}}\delta_n\rp{q_{\;\!\!n}}\gamma_n
\lpp{\
q''}\!\gamma_0\lp{q_{\;\!\!1}}\delta_1\rp{q_{\;\!\!1}}\gamma_1\cdots\lp{q_{\;\!\!n}}\delta_n\rp{q_{\;\!\!n}}\gamma_n\rpp{\;q''}
\gamma_0\lp{q_{\;\!\!1}}\delta_1\rp{q_{\;\!\!1}}\gamma_1\cdots\lp{q_{\;\!\!n}}\delta_n\rp{q_{\;\!\!n}}\gamma_n\\

\rightarrow&\hspace*{-2mm}\gamma_0\lp{q_{\;\!\!1}}\delta_1\rp{q_{\;\!\!1}}\gamma_1\cdots
\lp{q_{\;\!\!n}}\delta_n\rp{q_{\;\!\!n}}\gamma_n \gamma_0\lp{q_{\;\!\!1}}\delta_1\rp{q_{\;\!\!1}\;}\ \;\!\lpp{\
q''}\!\gamma_1\cdots
 \lp{q_{\;\!\!n}}\delta_n\rp{q_{\;\!\!n}}\gamma_n
\gamma_0\lp{q_{\;\!\!1}}\delta_1\rp{q_{\;\!\!1}\;}\ \;\!
\rpp{\;q''}\gamma_1\cdots\lp{q_{\;\!\!n}}\delta_n\rp{q_{\;\!\!n}}\gamma_n\\

\stackrel{*\;}{\rightarrow}&\hspace*{-2mm}\gamma_0\lp{q_{\;\!\!1}}\delta_1\rp{q_{\;\!\!1}}\gamma_1\cdots
\lp{\;r_{\;\!\!n}}\delta_n\rp{r_{\;\!\!n}}\sigma \lpp{\ q''}\!\gamma_1\cdots
 \lp{\;r_{\;\!\!n}}\delta_n\rp{r_{\;\!\!n}}\sigma
 \rpp{\;q''}\gamma_1\cdots\lp{q_{\;\!\!n}}\delta_n\rp{q_{\;\!\!n}}\gamma_n\\

 \rightarrow&\hspace*{-2mm}\gamma_0\lp{q_{\;\!\!1}}\delta_1\rp{q_{\;\!\!1}}\ \;\! \lpp{\
q'}\!\gamma_1\cdots
 \lp{\;r_{\;\!\!n}}\delta_n\rp{r_{\;\!\!n}}\sigma
 \rpp{q'}\gamma_1\cdots\lp{q_{\;\!\!n}}\delta_n\rp{q_{\;\!\!n}}\gamma_n.
\end{array}$$
This last term is clearly in semi-canonical form. On the other hand it verifies the properties of the term
$\pi_1$ in~\ref{item:limit_terms_reduction-b}. Suppose now that $\alpha$ is of type (II). If $q\not\in\{-1,1\}$,
then $q$ has some prime divisor $p$. Let $k=\frac{q}{p}$. Applying an expansion of type 1 to $\pi$ one gets the
term $\lp{k}\rho^p\rp{k}\;$. Hence, $\pi$ reduces to the term $\lp{k}\tau\rp{k}\;$ where $\tau$ is the
canonical form of $\rho^p$. By Lemma~\ref{lemma:normalization_of_product}~\ref{item:normalization_of_product-2},
as $\alpha$ is of type (II), the lt-length of $\tau$ is exactly $pn$ and thus greater than $1$. Therefore, by
the case $n>1$ above, $\lp{k}\tau\rp{k}\;$, and on its turn  $\pi$, reduces to a term $\pi_1$ as stated
in~\ref{item:limit_terms_reduction-b}. Finally, for $q\in\{-1,1\}$, apply an expansion of type $4_R$ to $\pi$ in
order to obtain the term $\ \lp{q'}\rho\rp{q'}\;\rho$, where $q'=q-1\in\{-2,0\}$. By the previous cases, this
term reduces to some term   $\varepsilon_0\lp{r}\beta\rp{r}\varepsilon_1\rho$ with
$\varepsilon_0\lp{r}\beta\rp{r}\varepsilon_1$ a rank $i+1$  semi-canonical form such that
$\rk{\varepsilon_0}=\rk{\varepsilon_1}=i$. To complete the construction of $\pi_1$ in the current case it
suffices to put $\varepsilon_1\rho$ in canonical form,  thus showing that condition~\ref{item:limit_terms_reduction-b}
holds.
\end{proof}

We are now ready to present the first step of the canonical form reduction algorithm.

\paragraph{\textbf{Step 1.}} The procedure to compute an
equivalent semi-canonical form $\alpha_1$ of an arbitrary rank $i+1$ term $\alpha$ is as follows.
\begin{enumerate}[label=1.\arabic*)]
\item\label{step11} In case $i=0$, declare $\alpha_1$ to be
$\alpha$ and stop (since every rank $1$ term is already in semi-canonical form).

\item\label{step12} Apply the rank $i$ canonical form reduction algorithm to each base of $\alpha$. We note that, this way, some
 (or all) of the original rank $i+1$ limit terms may have been transformed
into terms with strictly smaller rank. If the term obtained is rank
$j+1$ with $j<i$, then go to the beginning of Step 1 and take $i$ as
$j$.

\item\label{step13} Replace each rank $i+1$ limit term $\pi$ by the
semi-canonical form $\pi_1$ given by Proposition~\ref{prop:limit_terms_reduction}. Once again, if the term
obtained is no longer of rank $i+1$, then go to the beginning of Step 1.

\item\label{step14} Apply the rank $i$ canonical form reduction algorithm to each
primary subterm that does not occur as a base.
\end{enumerate}

The $\bk$-term $\alpha_1$ that emerges from this procedure is indeed a term in semi-canonical form. This is an
immediate consequence of Proposition~\ref{prop:limit_terms_reduction}~\ref{item:limit_terms_reduction-b} since
the bases of $\alpha_1$ are the bases $\beta$ of the subterms
$\pi_1=\varepsilon_0\lp{r}\beta\rp{r}\varepsilon_1$, introduced on Step 1.3, that come from that result and, so,
are terms in circular canonical form. Moreover, as $\varepsilon_0$ and $\varepsilon_1$ are rank $i$ terms
in canonical form, the reduction made by Step 1.4 does not change the final portion of $\varepsilon_0$ neither
the initial portion of $\varepsilon_1$, except for possible modifications of the exponents of the corresponding
limit terms. As a result, the term obtained from $\alpha_1$ by replacing each limit term $\;\lp{r}\beta\rp{r}\;$
by $\beta^2$ is in canonical form, so that $\alpha_1$ is in semi-canonical form.

\section{Languages associated with $\bk$-terms}\label{section:expansions}
 An alternative proof of correctness  of McCammond's normal form reduction algorithm for
$\omega$-terms over $\A$ was presented in~\cite{Almeida&Costa&Zeitoun:2012} and is based on properties of
certain regular languages $L_{\mathtt n}(\alpha)$ associated with $\omega$-terms $\alpha$, where ${\mathtt n}$
is a positive integer. Informally, the language $L_{\mathtt n}(\alpha)$ is obtained from $\alpha$ by replacing
each $\omega$-power by a power of exponent at least ${\mathtt n}$. The key property of the languages $L_{\mathtt
n}(\alpha)$ is that they are star-free when $\alpha$ is in McCammond's normal form and ${\mathtt n}$ is
sufficiently large. In this paper, similar languages will play a fundamental role in the proof of
Theorem~\ref{theo:canonical_form}. Given a $\bk$-term $\alpha$  and a pair $({\mathtt n},{\mathtt p})$ of
positive integers, we define below a language $L_{{\mathtt n},{\mathtt p}}(\alpha)$ whose elements, informally
speaking, are obtained from $\alpha$ by  recursively replacing each $\omega$ by an integer beyond ${\mathtt n}$
and congruent modulo ${\mathtt p}$ with that threshold. In particular, when the $\bk$-term $\alpha$ is an
$\omega$-term, $L_{{\mathtt n},1}(\alpha)=L_{{\mathtt n}}(\alpha)$. So, the above operators $L_{{\mathtt n}}$
associated with $\omega$-terms constitute an instance of a more general concept of operators $L_{{\mathtt
n},{\mathtt p}}$ associated with $\bk$-terms. Moreover, as we shall see below, the basic properties of the
operators $L_{{\mathtt n}}$ presented in~\cite{Almeida&Costa&Zeitoun:2012} extend easily to $L_{{\mathtt
n},{\mathtt p}}$.

\paragraph{\textbf{Expansions of $\bk$-terms.}} Let $\alpha$ be a $\bk$-term. Denote by $Q(\alpha)$
the set of all $q\in\Z$ for which there exists a subterm of $\alpha$
of the form $\beta^{\omega+q}$, that is,
$$Q(\alpha)=\{q\in\Z:\mbox{$\lp{q}$ occurs in the well-parenthesized
word of $A_\Z$ representing $\alpha$}\}.$$ Now, let $\nu(\alpha)$ be the nonnegative integer
$$\nu(\alpha)=\mbox{max}\{|q|:q\in Q(\alpha)\},$$
named the \emph{scale of $\alpha$}, and note the following immediate
property of this parameter.
\begin{remark}\label{remark:parameter_varsigma}
 If $\alpha'$ is either a subterm or an expansion of a $\bk$-term $\alpha$, then
$\nu(\alpha')\leq \nu(\alpha)$.
\end{remark}

Fix a  pair of positive integers $({\mathtt n},{\mathtt p})$. Usually we will impose high lower bounds for such
integers in order to secure the properties we need. For now, when the pair $({\mathtt n},{\mathtt p})$ is
associated with a $\bk$-term $\alpha$, we assume that ${\mathtt n}$ is greater than the scale
$\nu(\alpha)$ of $\alpha$. For each $q\in Q(\alpha)$, we let $\overline{q}$ be the set
\begin{equation}\label{eq:def_overline_q}
\overline{q}=\{{\mathtt n}+j{\mathtt p}+q:j\geq 0\}
\end{equation} of positive integers
with minimal element ${\mathtt n}+q$ and congruent mod ${\mathtt
p}$.

 The language $L_{{\mathtt n},{\mathtt p}}(\alpha)$ is formally defined as
follows, by means of sequential expansions that unfold the outermost
$(\omega+q)$-powers enclosing subterms of maximum rank.

\begin{definition}[Word expansions] When
$\alpha\in A^*$, we let $E_{{\mathtt n},{\mathtt p}}(\alpha)=\{\alpha\}$. Otherwise, let
$\alpha=\gamma_0\delta_1^{\omega+q_1}\gamma_1\cdots \delta_r^{\omega+q_r}\gamma_r$ be a rank $i+1$ $\bk$-term,
where $\rk {\delta_k}=i$ and $\rk{\gamma_j}\leq i$ for all $k$ and
 $j$. We let
 $$E_{{\mathtt n},{\mathtt p}}(\alpha) =\{\gamma_0\delta_1^{n_1}\gamma_1\cdots\delta_r^{n_r}\gamma_r:\mbox{ $n_j\in\overline{q_j}$ for $j=1,\ldots,r$}\}.$$
  For a set $K$ of $\bk$-terms, we let
$E_{{\mathtt n},{\mathtt p}}(K)=\bigcup_{\beta\in
  K}E_{{\mathtt n},{\mathtt p}}(\beta)$. We then let $$L_{{\mathtt n},{\mathtt p}}(\alpha)=E_{{\mathtt n},{\mathtt p}}^{\rk
  \alpha}(\alpha),$$  where $E_{{\mathtt n},{\mathtt p}}^j$ is the $j$-fold iteration of the operator
  $E_{{\mathtt n},{\mathtt p}}$.
\end{definition}

  For example, let $\alpha=(ab)^{\omega -1}aaba^{\omega}b(ab)^{\omega+5}a$ and let
$({\mathtt n},{\mathtt p})$ be arbitrary. We have
$$L_{{\mathtt n},{\mathtt p}}(\alpha)=E_{{\mathtt n},{\mathtt p}}(\alpha)=\{(ab)^{{\mathtt n}+j{\mathtt p}-1}aaba^{{\mathtt n}
+k{\mathtt p}}b(ab)^{{\mathtt n}+\ell {\mathtt p}+5}a:j,k,\ell\geq 0\}.$$  Consider now the rank $2$ canonical
form $\beta=(a^{\omega-1}b)^{\omega}a^{\omega+1}$. We have $L_{8,4}(\beta)=E_{8,4}^2(\beta)$ and
$E_{8,4}(\beta)=\{(a^{\omega-1 }b)^{8+4j}a^{\omega+1}:j\geq 0\}$. Hence
$$\begin{array}{rl}
L_{8,4}(\beta)&=\bigcup_{j\geq 0}E_{8,4}((a^{\omega-1 }b)^{8+4j}a^{\omega+1})\\[1mm]
&=\bigcup_{j\geq 0}\{a^{7+4k_1}ba^{7+4k_2}b\cdots a^{7+4k_{8+4j}}ba^{9+4k_{9+4j}}:k_1,\ldots,k_{9+4j}\geq
0\}.\end{array}$$

The next lemma is analogous to~\cite[Lemma~3.2]{Almeida&Costa&Zeitoun:2012} and presents some simple properties
of the operators $E_{{\mathtt n},{\mathtt p}}$ and $L_{{\mathtt n},{\mathtt p}}$.
\begin{lemma}\label{lemma:Lnp}
 Let $({\mathtt n},{\mathtt p})$ be a pair of positive integers. The following formulas hold,
 where we assume that ${\mathtt n}$ is greater than the scale of all $\bk$-terms involved:
  \begin{enumerate}
  \item\label{item:Ln-1} for $\bk$-terms  $\alpha$ and $\beta$,
    $$E_{{\mathtt n},{\mathtt p}}(\alpha\beta)=
    \begin{cases}
      E_{{\mathtt n},{\mathtt p}}(\alpha)\,E_{{\mathtt n},{\mathtt p}}(\beta) & \mbox{if $\rk \alpha=\rk \beta$}\\
      \alpha\,E_{{\mathtt n},{\mathtt p}}(\beta) & \mbox{if $\rk \alpha<\rk \beta$}\\
      E_{{\mathtt n},{\mathtt p}}(\alpha)\,\beta & \mbox{if $\rk \alpha>\rk \beta$};
    \end{cases}
    $$
  \item\label{item:Ln-1b} for a $\bk$-term $\alpha$, $L_{{\mathtt n},{\mathtt p}}(\alpha)=L_{{\mathtt n},{\mathtt p}}
  (E_{{\mathtt n},{\mathtt p}}(\alpha))$;
  \item\label{item:Ln-2} for sets $U$ and $V$ of $\bk$-terms, we have
  $L_{{\mathtt n},{\mathtt p}}(UV)=L_{{\mathtt n},{\mathtt p}}(U)\,L_{{\mathtt n},{\mathtt p}}(V)$;
  \item\label{item:Ln-3} for a $\bk$-term $\alpha=\gamma_0\delta_1^{\omega+q_1}\gamma_1\cdots \delta_r^{\omega+q_r}\gamma_r$
  with each $\rk {\delta_k}=i$ and $\rk{\gamma_j}\leq i$,
    $$L_{{\mathtt n},{\mathtt p}}(\alpha)=L_{{\mathtt n},{\mathtt p}}(\gamma_0)\,L_{{\mathtt n},{\mathtt p}}(\delta_1^{\omega+q_1})\,
    L_{{\mathtt n},{\mathtt p}}(\gamma_1)\cdots
    L_{{\mathtt n},{\mathtt p}}(\delta_r^{\omega+q_r})\, L_{{\mathtt n},{\mathtt p}}(\gamma_r);$$
  \item\label{item:Ln-4} for a $\bk$-term $\alpha$ and an integer $q$,
    $L_{{\mathtt n},{\mathtt p}}(\alpha^{\omega+q})= L_{{\mathtt n},{\mathtt p}}(\alpha)^{{\mathtt n}+q}(L_{{\mathtt n},{\mathtt p}}(\alpha)^{\mathtt p})^*$.
  \end{enumerate}
\end{lemma}
\begin{proof} The proof of each condition~\ref{item:Ln-1}--\ref{item:Ln-3} is identical to the proof of the
corresponding statement in~\cite[Lemma~3.2]{Almeida&Costa&Zeitoun:2012}. For~\ref{item:Ln-4}, we only need to
introduce minor changes. We have
  $$\begin{array}{rl}
  L_{{\mathtt n},{\mathtt p}}(\alpha^{\omega+q})&
  \subrel{\hbox{\scriptsize\ref{item:Ln-1b}}}{=}
  L_{{\mathtt n},{\mathtt p}}(E_{{\mathtt n},{\mathtt p}}(\alpha^{\omega+q}))
  =\bigcup_{j\geq 0}L_{{\mathtt n},{\mathtt p}}(\alpha^{{\mathtt n}+j{\mathtt p}+q})\\[1mm]
  &\subrel{\hbox{\scriptsize\ref{item:Ln-2}}}{=}
  \bigcup_{j\geq 0}L_{{\mathtt n},{\mathtt p}}(\alpha)^{{\mathtt n}+j{\mathtt p}+q}
  =L_{{\mathtt n},{\mathtt p}}(\alpha)^{{\mathtt n}+q}(L_{{\mathtt n},{\mathtt p}}(\alpha)^{{\mathtt p}})^*,\end{array}$$
thus completing the proof of the lemma.
\end{proof}

The following important property of the languages $L_{{\mathtt n},{\mathtt p}}(\alpha)$ can now be easily
deduced.
\begin{proposition}\label{prop:Lnp_regular}
Let $\alpha$ be a $\bk$-term in canonical form and let  $({\mathtt n},{\mathtt p})$ be a pair of positive
integers with ${\mathtt n}>\nu(\alpha)$. Then  $L_{{\mathtt n},{\mathtt p}}(\alpha)$ is a regular
language.
\end{proposition}
\begin{proof} We proceed by induction on
$\rk{\alpha}$. For $\rk{\alpha}=0$ the result is clear since in this case $L_{{\mathtt n},{\mathtt
p}}(\alpha)=\{\alpha\}$. Let now $\rk {\alpha}=i+1$ with $i\geq 0$ and suppose, by the induction hypothesis, that
the lemma holds for $\bk$-terms of rank at most $i$. Let $\alpha=\gamma_0\delta_1^{\omega+p_1}\gamma_1\cdots
\delta_r^{\omega+p_r}\gamma_r$ be the canonical form expression for $\alpha$, where $\rk {\delta_k}=i$ and
$\rk{\gamma_j}\leq i$ for all $k$ and $j$. Then, by Lemma~\ref{lemma:Lnp},
$$L_{{\mathtt n},{\mathtt p}}(\alpha)= L_{{\mathtt n},{\mathtt p}}(\gamma_0)
L_{{\mathtt n},{\mathtt p}}(\delta_1)^{{\mathtt n}+q_1}(L_{{\mathtt n},{\mathtt p}}(\delta_1)^{{\mathtt
p}})^*L_{{\mathtt n},{\mathtt p}}(\gamma_1)\cdots L_{{\mathtt n},{\mathtt p}}(\delta_r)^{{\mathtt
n}+q_r}(L_{{\mathtt n},{\mathtt p}}(\delta_r)^{{\mathtt p}})^*L_{{\mathtt n},{\mathtt p}}(\gamma_r).$$ By the
induction hypothesis, each $L_{{\mathtt n},{\mathtt p}}(\gamma_j)$ and $L_{{\mathtt n},{\mathtt p}}(\delta_k)$
is a regular language, whence $L_{{\mathtt n},{\mathtt p}}(\alpha)$ is itself a regular language. This completes
the inductive step and concludes the proof of the result.
\end{proof}

For instance, the language $L_{8,4}(\beta)$ associated with the above $\bk$-term
$\beta=(a^{\omega-1}b)^{\omega}a^{\omega+1}$  admits the regular expression
$L_{8,4}(\beta)=(a^7(a^4)^*b)^8((a^7(a^4)^*b)^4)^*a^9(a^4)^*$. Notice that, in this example, ${\mathtt n}=8$ is a multiple of ${\mathtt p}=4$ and so the sequence $\bigl((a^{k!-1}b)^{k!}a^{k!+1}\bigr)_k$ of $A^+$ is ultimately contained in $L_{8,4}(\beta)$. Thus,
\begin{equation}\label{ex:etabeta_clL}
\eta(\beta)\in\mbox{cl}(L_{8,4}(\beta))
\end{equation}
where $\eta:T_A^{\bk}\rightarrow\omek{\bk}{A}{\Se}$ is the homomorphism of $\bk$-semigroups that sends each  $x\in A$ to itself and $\mbox{cl}(L_{8,4}(\beta))$ denotes the topological closure of the language $L_{8,4}(\beta)$ in $\OAS$.

\paragraph{\textbf{Schemes for canonical forms.}} We define the \emph{length} of a $\bk$-term $\alpha$ as the length
of the corresponding well-parenthesized word over $A_\Z$, and denote it $|\alpha|$. We now associate to each
$\bk$-term $\alpha$ a parameter $\mu(\alpha)$, introduced in~\cite{Almeida&Costa&Zeitoun:2012} for
$\omega$-terms. In case $\alpha\in A^+$, let
  $\mu(\alpha)=0$. Otherwise,~let
  $$\mu(\alpha)=2^{\rk \alpha}\max\{|\beta|:
  \beta\text{ is a crucial portion of $\alpha^2$}\}.$$
It is important to remark the following feature of this parameter,
whose proof is an easy adaptation
of~\cite[Lemma~3.5]{Almeida&Costa&Zeitoun:2012}.
\begin{lemma}\label{lemma:parameter_mu}
 If $\alpha'$ is an expansion of a $\bk$-term $\alpha$, then
$\mu(\alpha')\leq \mu(\alpha)$.
\end{lemma}

 Let $\alpha$ be a $\bk$-term in canonical form and let $({\mathtt n},{\mathtt p})$ be a pair of integers. We say that $({\mathtt n},{\mathtt p})$
 is a \emph{scheme for} $\alpha$ if the following conditions hold:
\begin{itemize}
\item ${\mathtt n}$ is a multiple of ${\mathtt p}$ such that ${\mathtt n}-{\mathtt
p}>\mu(\alpha)$;
\item ${\mathtt p}>2\nu(\alpha)$.
\end{itemize}

The next result is an immediate consequence of
Proposition~\ref{prop:canonicalform_vs_subterms}, of
Remark~\ref{remark:parameter_varsigma} and of
Lemma~\ref{lemma:parameter_mu}.
\begin{lemma}\label{lemma:table-vs_subterm_and_expansion}
Let $\alpha$ be a $\bk$-term in canonical form and let $({\mathtt n},{\mathtt p})$ be a scheme for $\alpha$. If
$\alpha'$ is  an expansion of $\alpha$, then $\alpha'$ is in canonical form and $({\mathtt n},{\mathtt p})$ is a
scheme  for $\alpha'$.
\end{lemma}

 The following is a significant property of a scheme.
\begin{lemma}\label{lemma:eqnormal_forms_vs_eqexpansions_eqrank}
Let $\alpha$ and $\beta$ be canonical forms with $\rk{\alpha}=\rk{\beta}$ and let $({\mathtt n},{\mathtt p})$ be
a scheme for both $\alpha$ and $\beta$. If $L_{{\mathtt n},{\mathtt p}}(\alpha)\cap L_{{\mathtt n},{\mathtt
p}}(\beta)\neq\emptyset$, then $\alpha=\beta$.
\end{lemma}
\begin{proof} The proof is made by induction on the rank of  $\alpha$ (and $\beta$). The case $\rk{\alpha}=0$ is
trivial. Indeed, in this case we have $L_{{\mathtt n},{\mathtt
p}}(\alpha)=\{\alpha\}$ and $L_{{\mathtt n},{\mathtt
p}}(\beta)=\{\beta\}$.

Let now $\rk{\alpha}=i+1$ with $i\geq 0$ and suppose, by induction hypothesis, that the lemma holds for rank $i$
canonical forms. Let $\alpha=\gamma_0\delta_1^{\omega+p_1}\gamma_1\cdots \delta_r^{\omega+p_r}\gamma_r$ and
$\beta=\pi_0\rho_1^{\omega+q_1}\pi_1\cdots \rho_s^{\omega+q_s}\pi_s$ be the canonical form expressions for
$\alpha$ and $\beta$  and suppose that $w\in L_{{\mathtt n},{\mathtt p}}(\alpha)\cap L_{{\mathtt n},{\mathtt
p}}(\beta)$. By Lemma~\ref{lemma:Lnp}~\ref{item:Ln-1b} there exist expansions $\alpha'\in E_{{\mathtt
n},{\mathtt p}}(\alpha)$ of $\alpha$ and $\beta'\in E_{{\mathtt n},{\mathtt p}}(\beta)$ of $\beta$ such that
$w\in L_{{\mathtt n},{\mathtt p}}(\alpha')\cap L_{{\mathtt n},{\mathtt p}}(\beta')$. In particular, $\alpha'$
and $\beta'$ are rank $i$ canonical forms of the type $\alpha'=\gamma_0\delta_1^{m_1}\gamma_1\cdots
\delta_r^{m_r}\gamma_r$, with each $m_\ell\in \overline{p_\ell}=\{{\mathtt n}+j{\mathtt p}+p_\ell:j\geq 0\}$,
say $m_\ell={\mathtt n}+j_\ell{\mathtt p}+p_\ell$ with $j_\ell\geq 0$, and $\beta'=\pi_0\rho_1^{n_1}\pi_1\cdots
\rho_s^{n_s}\pi_s$, where each $n_\ell={\mathtt n}+k_\ell{\mathtt p}+q_\ell$ with $k_\ell\geq 0$. Moreover, by
Lemma~\ref{lemma:table-vs_subterm_and_expansion}, $({\mathtt n},{\mathtt p})$ is a scheme  for both $\alpha'$
and $\beta'$. Hence, the induction hypothesis entails the equality of the $\bk$-terms $\alpha'$ and $\beta'$,
that is,
\begin{equation}\label{eq:varepsilon=embeta}
\gamma_0\delta_1^{m_1}\gamma_1\cdots \delta_r^{m_r}\gamma_r=\pi_0\rho_1^{n_1}\pi_1\cdots \rho_s^{n_s}\pi_s.
\end{equation}

As $({\mathtt n},{\mathtt p})$ is a scheme for $\alpha$ and $\beta$, $m_\ell\geq {\mathtt n}+p_\ell\geq {\mathtt
n}-\nu(\alpha)>{\mathtt n}-{\mathtt p}$ and, analogously, $n_\ell>{\mathtt n}-{\mathtt p}$ for every
$\ell$. On the other hand ${\mathtt n}-{\mathtt p}>\max\{\mu(\alpha),\mu(\beta)\}$. Thus, in particular, $m_1$
and $n_1$ are both greater than $\max\{|\delta_r^{\omega+ p_r}\gamma_r\gamma_0\delta_1^{\omega
+p_1}|,|\rho_s^{\omega+ q_s}\pi_s\pi_0\rho_1^{\omega +q_1}|\}$.
  Hence, the terms $\delta_1^{m_1}$ and
$\rho_1^{n_1}$, occurring on the opposite sides of  equality~\eqref{eq:varepsilon=embeta}, must overlap on a
factor of length at least $|\delta_1|+|\rho_1|$. Therefore, by Fine and Wilf's Theorem, $\delta_1$ and $\rho_1$
have conjugates that are powers of the same $\bk$-term, say $\sigma$.  Since $\delta_1$ and $\rho_1$ are Lyndon terms by
condition~\ref{item:cf-2} of the $\bk$-term canonical form definition, it follows that $\delta_1=\sigma=\rho_1$.
Suppose, without loss of generality, that $|\gamma_0|\geq |\pi_0|$ and recall that any Lyndon term is
unbordered. Then, as $\gamma_0\delta_1$ is a prefix of $\beta'$, $\gamma_0\delta_1=\pi_0\delta_1^j$ for some
$j\geq 1$. Hence $\gamma_0=\pi_0$ since $\alpha$ is in canonical form and condition~\ref{item:cf-3} of the
$\bk$-term canonical form definition states that $\delta_1$ is not a suffix of $\gamma_0$. On the other hand,
the canonical forms $\alpha$ and $\beta$ verify condition~\ref{item:cf-4}. Thus, $\delta_1$ is not a prefix of
$\gamma_1\delta_2^{m_2}$ and $\rho_1$ is not a prefix of $\pi_1\rho_2^{n_2}$. Consequently, the equalities
$\alpha'=\beta'$, $\gamma_0=\pi_0$ and $\delta_1=\rho_1$ and the fact that both $m_2$ and $n_2$ are greater than
$\max\{\mu(\alpha),\mu(\beta)\}$ (and so greater than $|\delta_1|$) imply that $m_1=n_1$. As ${\mathtt n}$ is a
multiple of ${\mathtt p}$ by the definition of a scheme, the positive integers $m_1$ and $n_1$ are congruent mod
${\mathtt p}$ with $p_1$ and $q_1$ respectively. Therefore $p_1=q_1$ once ${\mathtt
p}>2\max\{\nu(\alpha),\nu(\beta)\}$. Iterating the above procedure, one deduces that, for every
$1\leq \ell\leq\mbox{min}\{r,s\}$, $\gamma_{\ell-1}=\pi_{\ell-1}$, $\delta_{\ell}=\rho_{\ell}$,
$p_{\ell}=q_{\ell}$ and
$$\gamma_{\ell-1}\delta_{\ell}^{m_\ell}\gamma_{\ell}\cdots
\delta_r^{m_r}\gamma_r=\pi_{\ell-1}\rho_\ell^{n_\ell}\pi_\ell\cdots \rho_s^{n_s}\pi_s.$$ By symmetry, we have
further that $\gamma_r=\pi_s$. Now, since each $m_k$ and each $n_\ell$ is greater than
$\max\{\mu(\alpha),\mu(\beta)\}$, it is now straightforward to deduce that $r=s$. This shows that
$\alpha=\beta$ and concludes the inductive step of the proof.
\end{proof}

The following result is an extension of~\cite[Theorem 5.3]{Almeida&Costa&Zeitoun:2012} and it will be essential
to prove Theorem~\ref{theo:canonical_form}.
\begin{proposition}\label{prop:eq_normal_forms_vs_eq_expansions}
Let $\alpha$ and $\beta$ be canonical forms  and let $({\mathtt n},{\mathtt p})$ be a scheme for both $\alpha$
and $\beta$ such that ${\mathtt n}-{\mathtt p}> \max\{|\alpha|,|\beta|\}$. If $L_{{\mathtt n},{\mathtt
p}}(\alpha)\cap L_{{\mathtt n},{\mathtt p}}(\beta)\neq\emptyset$, then $\alpha=\beta$.
\end{proposition}
\begin{proof} Let $w\in L_{{\mathtt n},{\mathtt p}}(\alpha)\cap L_{{\mathtt n},{\mathtt
p}}(\beta)$.  Suppose that $\rk{\alpha}> \rk{\beta}=i$ and let $j=\rk{\alpha}-i$. Hence, by the hypothesis $w\in
L_{{\mathtt n},{\mathtt p}}(\alpha)$ and Lemma~\ref{lemma:Lnp}~\ref{item:Ln-1b}, there is $\alpha'\in
E_{{\mathtt n},{\mathtt p}}^j(\alpha)$ such that $w\in L_{{\mathtt n},{\mathtt p}}(\alpha')$. Moreover,
$\alpha'$ is a rank $i$ canonical form and $({\mathtt n},{\mathtt p})$ is a scheme  for
 $\alpha'$. Therefore, by Lemma~\ref{lemma:eqnormal_forms_vs_eqexpansions_eqrank},
$\alpha'=\beta$. This is however impossible since $|\beta|<{\mathtt n}-{\mathtt p}$, by hypothesis, and
${\mathtt n}-{\mathtt p}<|\alpha'|$, by the fact that $\alpha'\in E_{{\mathtt n},{\mathtt p}}^j(\alpha)$ and
${\mathtt n}-{\mathtt p}<{\mathtt n}-\nu(\alpha)\leq {\mathtt n}+q=\min \overline{q}$ for every $q\in
Q(\alpha)$. By symmetry it follows that $\rk{\alpha}= \rk{\beta}$ and so, by
Lemma~\ref{lemma:eqnormal_forms_vs_eqexpansions_eqrank}, $\alpha=\beta$.
\end{proof}
\section{Main results}\label{section:main results}
 For $L\subseteq A^+$, let $\mbox{cl}(L)$ be the topological
closure of $L$ in $\OAS$ and notice that $\mbox{cl}(L)\cap A^+=L$ since any sequence of words converging  to a
word  $w\in A^+$ is ultimately equal to $w$.  We can now complete the proof of our central result.
\begin{theorem}\label{theo:canonical_form}
Let $\alpha$ and $\beta$ be $\bk$-terms in canonical form. If $\Se\models\alpha=\beta$, then $\alpha$ and
$\beta$ are the same $\bar\kappa$-term.
\end{theorem}
\begin{proof} We adapt the corresponding proof for McCammond's normal forms, given in~\cite[Corollary~5.4]{Almeida&Costa&Zeitoun:2012}. Let $({\mathtt n},{\mathtt
p})$ be a scheme for both $\alpha$ and $\beta$, with ${\mathtt n}-{\mathtt p}> \max\{|\alpha|,|\beta|\}$. The
languages $L_{{\mathtt n},{\mathtt p}}(\alpha)$ and $L_{{\mathtt n},{\mathtt p}}(\beta)$ are regular by
Proposition~\ref{prop:Lnp_regular}, whence $\mbox{cl}(L_{{\mathtt n},{\mathtt p}}(\alpha))$ and
$\mbox{cl}(L_{{\mathtt n},{\mathtt p}}(\beta))$ are clopen subsets of $\OAS$. On the other hand, since ${\mathtt n}$ is a
multiple of ${\mathtt p}$ by the definition of a scheme, and as exemplified in~\eqref{ex:etabeta_clL} ,
$\eta(\alpha)\in\mbox{cl}(L_{{\mathtt n},{\mathtt p}}(\alpha))$ and $\eta(\beta)\in\mbox{cl}(L_{{\mathtt
n},{\mathtt p}}(\beta))$. As $\eta(\alpha)=\eta(\beta)$ by hypothesis, it follows that
$\mbox{cl}(L_{{\mathtt n},{\mathtt p}}(\alpha))\cap\mbox{cl}(L_{{\mathtt n},{\mathtt p}}(\beta))$ is a nonempty
open set and so it contains some elements of the dense set~$A^+$. Since $\mbox{cl}(L_{{\mathtt n},{\mathtt
p}}(\alpha))\cap\mbox{cl}(L_{{\mathtt n},{\mathtt p}}(\beta))\cap A^+=L_{{\mathtt n},{\mathtt p}}(\alpha)\cap
L_{{\mathtt n},{\mathtt p}}(\beta)$, we deduce that $L_{{\mathtt n},{\mathtt p}}(\alpha)\cap L_{{\mathtt
n},{\mathtt p}}(\beta)\neq\emptyset$. Hence $\alpha=\beta$ by
Proposition~\ref{prop:eq_normal_forms_vs_eq_expansions}.
\end{proof}

In particular, we derive from this result that the canonical form reduction algorithm
applied to any $\bk$-term produces a unique $\bk$-term in canonical form. It also leads to an easy deduction of
the main results of this paper.
\begin{theorem}\label{theo:decid_k_word_problem}
 The $\bk$-word problem for  ${\bf S}$ is decidable. More precisely, given $\bk$-terms $\alpha$ and $\beta$, the canonical form
reduction algorithm can be used to decide whether ${\bf S}$
satisfies $\alpha=\beta$.
\end{theorem}
\begin{proof} Let $\alpha'$ and $\beta'$ be canonical forms obtained,
respectively, from $\alpha$ and $\beta$ by the application of the canonical form reduction algorithm. By
construction of the algorithm, ${\bf S}$ verifies $\alpha=\alpha'$ and $\beta=\beta'$. In view of
Theorem~\ref{theo:canonical_form}, to decide whether $\Se$
 verifies $\alpha=\beta$ it suffices therefore to verify whether $\alpha'$ and $\beta'$ are the same $\bk$-term.
\end{proof}

\begin{theorem}\label{theo:kappa_basis_S}
 The set $\Sigma$ is a basis of $\bk$-identities for $\Se^{\bk}$, the $\bk$-variety generated by all finite semigroups.
\end{theorem}
\begin{proof} Recall that the rewriting rules used in the canonical form reduction
algorithm are determined by the $\bk$-identities of $\Sigma$. Hence, it suffices to prove that, for all
$\bk$-terms $\alpha$ and $\beta$, $\Se\models \alpha=\beta$ if and only if $\alpha\sim \beta$. That $\alpha\sim
\beta$ implies $\Se\models \alpha=\beta$ follows from the fact that $\Se$ verifies all the $\bk$-identities of
$\Sigma$. To show the reverse implication, suppose that $\Se\models \alpha=\beta$ and let $\alpha'$ and $\beta'$
be the canonical forms of $\alpha$ and $\beta$. As ${\bf S}$ verifies $\alpha=\alpha'$ and $\beta=\beta'$, it
also verifies $\alpha'=\beta'$. By Theorem~\ref{theo:canonical_form} we deduce that $\alpha'=\beta'$.  Since
$\alpha\sim\alpha'$ and $\beta\sim\beta'$ it follows by transitivity that $\alpha\sim\beta$.
\end{proof}

The instance of Theorem~\ref{theo:canonical_form} in which $\alpha$ and $\beta$ have rank at most 1 was proved,
in a different way, by the author together with Nogueira and Teixeira in~\cite{Costa&Nogueira&Teixeira:2013}.
Moreover, we have shown in that paper that the pseudovariety $\LG$ does not identify different canonical forms
of rank at most $1$. It is, however, well-known that $\LG$ identifies the canonical forms $(a^\omega b)^\omega
a^\omega$ and $a^\omega$. This remark suggests the introduction of the notion of {\em $\bk$-index of a
pseudovariety $\V$}, denoted $i_{\bk}(\V)$, as being:  the least integer $j\geq 0$, whenever it exists, such that $\V$
identifies two different canonical forms of rank at most $j$; $+\infty$, otherwise. So, $i_{\bk}(\LG)=2$ and
$i_{\bk}(\Se)=+\infty$. The pseudovarieties of $\bk$-index $0$ are, by definition, the ones that verify some
nontrivial identity. For easy examples of $\bk$-index $1$, we may refer the pseudovarieties  $\G$ of groups,
$\Nil$ of nilpotent semigroups and $\A$. As $\Nil\subseteq \A$, it follows that  $i_{\bk}(\V)=1$ for every
aperiodic pseudovariety $\V$ containing $\Nil$. For an integer $j>2$, the author is not aware of examples of
pseudovarieties having $\bk$-index $j$.

An unary implicit signature is a signature formed by unary non-explicit implicit operations together with
multiplication. For instance, the signatures $\omega$, $\kappa$ and $\bk$ are unary. The above notion can be
extended to any unary implicit signature $\sigma$, for which there is a natural definition of rank for
$\sigma$-terms, as follows. For a pseudovariety $\V$, let $$\begin{array}{rl} I_\sigma(\V)=\{j\geq
0:&\mbox{\hspace*{-2mm}there exist
$\sigma$-terms $\alpha$ and $\beta$ with rank at most $j$ }\\
&\mbox{\hspace*{-2mm}such that $\V\models \alpha=\beta$ and $\Se\;\not\!\models \alpha=\beta$}\}.
\end{array}$$
The {\em $\sigma$-index of $\V$}, denoted $i_\sigma(\V)$, is defined to be $\min I_\sigma(\V)$ when $I_\sigma(\V)$ is non-empty
and to be $+\infty$ otherwise.

\section{Canonical representatives for $\omega$-terms over $\A$}\label{section:representatives_for_w-terms_overA}
In this section, we explain how the above results can be adjusted in order to obtain canonical representatives
for each class of $\omega$-terms with the same interpretation on each finite aperiodic semigroup.

\paragraph{\textbf{The canonical form algorithm for $\omega$-terms over $\A$.}} In Section~\ref{section:algorithm_normal_form}, we presented an algorithm that computes the canonical form of
any given $\bk$-term. In particular, for an $\omega$-term $\alpha$ the algorithm provides a unique $\bk$-term
$\alpha'$ in canonical form such that $\Se\models\alpha=\beta$ and, so, such that $\A\models\alpha=\beta$. As
far as the $\omega$-word problem over $\A$ is concerned, the trouble is that $\alpha'$ does not have to be an
$\omega$-term. In effect this is not a difficulty since in order to solve the word problem for $\kappa$-terms
over $\Se$, we also went outside the world of $\kappa$-terms. The real trouble is that $\omega$-terms with the
same value over $\A$ can have different canonical forms (when the $\omega$-terms are different over $\Se$). This is the case, for instance, of the $\omega$-terms
$a^\omega ab^\omega$ and $a^\omega bb^\omega$ whose canonical forms are respectively $a^{\omega+1} b^\omega$ and
$a^\omega b^{\omega+1}$.

An algorithm that computes, for each $\omega$-term $\alpha$, a unique  $\omega$-term $\alpha'$ in canonical
form with the same value over $\A$ can, however, be easily adapted from the algorithm in
Section~\ref{section:algorithm_normal_form}. The $\omega$-term $\alpha'$ will then be called the {\em canonical
form of $\alpha$ over $\A$}. For that, it suffices to replace everywhere in the algorithm each
occurrence of a symbol  $\;\lp{q}\;$ or $\;\rp{q}\;$ by, respectively, $\;\lp{0}\;$ and $\;\rp{0}\;$. This way
all terms involved are $\omega$-terms and this new algorithm preserves the value of the original $\omega$-term $\alpha$ over $\A$. Indeed, the elementary
changes are determined by the following rules for $\omega$-terms, obtained  from the rewriting rules for $\bk$-terms of Section~\ref{section:normal_forms} by the replacement of the symbols  $\;\lp{q}\;$ and $\;\rp{q}\;$ by $\;\lp{0}\;$ and $\;\rp{0}\;$,
\begin{xalignat*}{2}
&1.~\lpp{0\;\!}\ \lp{\;\!0}\alpha\rp{0\;\!}\ \rpp{0}\ \ \rightleftarrows\ \ \lp{0}\alpha\rp{0}&
&4_R.~ \lp{0}\alpha\rp{0}\alpha\  \rightleftarrows\ \ \; \lp{0}\alpha\rp{0}\\[-1mm]
&2.~ \lp{0}\alpha^n\rp{0}\ \ \rightleftarrows\ \ \lp{0}\alpha\rp{0}&
&4_L.~\, \alpha\lp{0}\alpha\rp{0}\ \ \rightleftarrows\ \ \; \lp{0}\alpha\rp{0} \\[-1mm]
&3.~ \lp{0}\alpha\rp{0}\ \;\lp{0}\alpha\rp{0}\ \ \rightleftarrows\ \  \;\! \lp{0}\alpha\rp{0}\ \ & &5.~\
\;\,\lp{0}\alpha\beta\rp{0}\alpha\ \rightleftarrows\ \alpha\lp{0}\beta\alpha\rp{0}
\end{xalignat*}
Actually, these are precisely the rules used in McCammond's
algorithm. Our algorithm for $\omega$-terms over $\A$ is essentially the same as McCammond's
algorithm except in the procedure to put the crucial portions in canonical form (in
view of their distinct definitions). For instance, the canonical form over $\A$ of the $\omega$-terms
$a^\omega ab^\omega$ and $a^\omega bb^\omega$ is $a^{\omega} b^\omega$, while their McCammond's normal form is $a^{\omega}ab b^\omega$.

\paragraph{\textbf{Star-freeness of the languages $L_{{\mathtt n},1}(\alpha)$.}} The star-freeness of
the languages $L_{\mathtt n}(\alpha)$, for $\omega$-terms $\alpha$ in McCammond's normal form and ${\mathtt n}$
large enough, was established in~\cite[Theorem~5.1]{Almeida&Costa&Zeitoun:2012}. For canonical forms an
identical property holds.
\begin{theorem}\label{theo:star-free}
  Let $\alpha$ be an $\omega$-term in canonical form and let ${\mathtt n}\geq\mu(\alpha)$. Then the
  language $L_{{\mathtt n},1}(\alpha)$ is star-free.
\end{theorem}

The proof of this result can be obtained by a mere adjustment of the
corresponding proof
of~\cite[Theorem~5.1]{Almeida&Costa&Zeitoun:2012} and, so, we do not
include it here. Actually, since each subterm of a canonical form is
in canonical form as well, the arguments can be usually simplified.
As a consequence of Theorem~\ref{theo:star-free} and of
 Proposition~\ref{prop:eq_normal_forms_vs_eq_expansions}, with  ${\mathtt
 p}=1$, one gets the following analogue of Theorem~\ref{theo:canonical_form}, that establishes the uniqueness of canonical forms for
$\omega$-terms over $\A$.
\begin{theorem}\label{theo:word-problem_over A}
Let $\alpha$ and $\beta$ be $\omega$-terms in canonical form. If
$\A\models\alpha=\beta$, then $\alpha=\beta$.
\end{theorem}

Once again, we omit the proof of this result since it is identical to the proof of the corresponding
result~\cite[Corollary~5.4]{Almeida&Costa&Zeitoun:2012} for McCammond's normal forms (and it is analogous to the
one of Theorem~\ref{theo:canonical_form}).

\section*{Acknowledgments}

I would like to thank Jorge Almeida for asking me about the decidability problem which is the object of this
paper. I also benefited greatly from our joint work with Marc Zeitoun about the fascinating McCammond's normal
forms.

This work was supported by the European Regional Development Fund, through the programme COMPETE, and by the
Portuguese Government through FCT -- \emph{Funda\c c\~ao para a Ci\^encia e a Tecnologia}, under the project
PEst-C/MAT/UI0013/2011.


\end{document}